\documentclass[english, 11pt]{amsart}
\usepackage{amssymb}
\usepackage{amsfonts}
\usepackage{amscd}
\usepackage[utf8x]{inputenc}
\usepackage[T1,T5]{fontenc}
\usepackage{xcolor}
\def\VF{\mathrm{VF}}
\def\Res{\mathrm{R}}
\def\VG{\mathrm{VG}}
\newcommand{\RF}{{\rm RF}}
\def\Smo{\mathrm{Smo}}
\def\WF{\mathrm{WF}}

\definecolor{immi}{rgb}{0,.5,0}

\renewcommand{\color}[1]{}

\long\def\blue#1{{\color{blue} #1}}

\newbox\removebox
\newcommand\remove[1]{%
\setbox\removebox=\ifmmode\hbox{$#1$}\else\hbox{#1}\fi%
\leavevmode
\rlap{\textcolor{blue}{\vrule height0.8ex depth-0.6ex width\wd\removebox}}%
\box\removebox
}
\long\def\bigremove#1{%
\par\setbox\removebox=\vbox{#1}%
\vbox{%
\vbox to0pt{\hbox{\tikz\draw[color=blue,thick] (0,0) -- (\wd\removebox,-\ht\removebox)  (\wd\removebox,0) -- (0,-\ht\removebox);}}
\box\removebox
}
}

\usepackage{mathrsfs} 

\newcommand{\cCexp}{\cC^{\mathrm{exp}}}

\newcommand{\grad}{\operatorname{grad}}

\newcommand{\Loc}{{\mathrm{Loc}}}

\def\ac{{\overline{\rm ac}}}
\def\Supp{\operatorname{Supp}}
\def\Sing{\operatorname{SS}}

\def\gLPas{\cL_{\rm gDP}}

\def\11{{\mathbf 1}}
\def\AA{{\mathbb A}}

\def\CC{{\mathbb C}}

\def\FF{{\mathbb F}}

\def\NN{{\mathbb N}}

\def\QQ{{\mathbb Q}}

\def\ZZ{{\mathbb Z}}

\def\cC{{\mathscr C}}
\def\cD{{\mathcal D}}

\def\cF{{\mathcal F}}

\def\cL{{\mathcal L}}
\def\cM{{\mathcal M}}

\def\cO{{\mathcal O}}

\def\cS{{\mathcal S}}

\def\cX{{\mathcal X}}
\def\cY{{\mathcal Y}}

\newtheorem{thm}[subsubsection]{Theorem}
\newtheorem{lem}[subsubsection]{Lemma}
\newtheorem{cor}[subsubsection]{Corollary}
\newtheorem{prop}[subsubsection]{Proposition}

\theoremstyle{definition}
\newtheorem{defn}[subsubsection]{Definition}
\newtheorem{example}[subsubsection]{Example}

\newtheorem{def-prop}[subsubsection]{Proposition-Definition}
\newtheorem{def-theorem}[subsubsection]{Theorem-Definition}
\newtheorem{def-lem}[subsubsection]{Lemma-Definition}

\theoremstyle{remark}
\newtheorem{remark}[subsubsection]{Remark}

\theoremstyle{plain}

\numberwithin{equation}{subsection}

  {\par\medskip\noindent #1\par\begingroup%
    \advance\leftskip by 1em\advance\rightskip by 1em}%
  {\par\endgroup}

\newcommand{\ord}{\operatorname{ord}}

\newcommand{\abs}[1]{\lvert#1\rvert}

\begin{document}

\setcounter{tocdepth}{1} 

\author[Cluckers]{Raf Cluckers}
\address{Universit\'e de Lille, Laboratoire Painlev\'e, CNRS - UMR 8524, Cit\'e Scientifique, 59655
Villeneuve d'Ascq Cedex, France, and,
KU Leuven, Department of Mathematics,
Celestijnenlaan 200B, B-3001 Leu\-ven, Bel\-gium}
\email{Raf.Cluckers@univ-lille.fr}
\urladdr{http://rcluckers.perso.math.cnrs.fr/}


\author[Halupczok]{Immanuel Halupczok}
\address{Mathematisches Institut,
Gebaeude 25.22, Universit\"atsstr. 1, 40225 D\"usseldorf,
Germany}
\email{math@karimmi.de}
\urladdr{http://www.immi.karimmi.de/en/}

\author[Loeser]{Fran\c cois Loeser}
\address{Sorbonne Universit\'e, UPMC Univ Paris 06, UMR 7586 CNRS, Institut Math\'ematique de Jussieu, F-75005 Paris, France}
\email{Francois.Loeser@imj-prg.fr}
\urladdr{https://webusers.imj-prg.fr/$\sim$francois.loeser/}

\author[Raibaut]{Michel Raibaut}
\address{Laboratoire de Math\'ematiques\\
Universit\'e Savoie Mont Blanc, B\^atiment Chablais, Campus Scientifique, Le Bourget du Lac, 73376 Cedex, France}
\email{Michel.Raibaut@univ-smb.fr}
\urladdr{www.lama.univ-savoie.fr/$\sim$raibaut/}







\title[Distributions and wave front sets, uniformly]
{Distributions and wave front sets in the uniform non-archimedean setting}

\begin{abstract}
We study some constructions on distributions in a uniform $p$-adic context, and also in large positive characteristic, using model theoretic methods. We introduce a class of distributions which we call distributions of $\cCexp$-class and which is based on the notion of $\cCexp$-class functions from \cite{CHallp}. This class of distributions is stable under Fourier transformation and has various forms of uniform behavior across non-archimedean local fields. We study wave front sets, pull-backs and push-forwards of distributions of this class. In particular we show that the wave front set is always equal to the complement of the zero locus of a $\cCexp$-class function.
We first revise and generalize some of the results of Heifetz that he developed in the $p$-adic context by analogy to results about real wave front sets by H\"ormander. In the final section, we study sizes of neighborhoods of local constancy of Schwartz-Bruhat functions and their push forwards in relation to discriminants.
\end{abstract}


\maketitle


\section{Introduction}

\subsection{}
 \blue{The study of wave front sets since H\"ormander has been a bridge between geometry and analysis, pure and applied, as in the study of partial differential equations and associated distributions. Our research is driven by the quest for $p$-adic and motivic analogues for results in real and complex geometry and analysis, where in the $p$-adic case the link between distributions and differential operators still has many mysterious aspects. Here, we introduce a uniform algebraic viewpoint on $p$-adic wave front sets, based on model theory, and we prove results like Theorem \ref{Smo} which seem to be waiting for real analogues, and which in particular yield uniformity in the local field and natural notions of families of distributions.}

\subsection{}In his paper
\cite{Hei85a}, Heifetz developed a $p$-adic version of the wave front set of a distribution first introduced by H\"ormander in the real case in
\cite{Hormander71}, as follows.
Let $K$ be a $p$-adic field and $X$ be a open subset of $K^n$.
A distribution $u$ on $X$ is an element of the linear dual of the complex vector space of locally constant functions with compact support on $X$.
For  $\Lambda$  an
open subgroup of $K^\times$  of finite index, one says that $u$ is $\Lambda$-smooth at
a point $(x_0,\xi_0)\in X \times (K^n \setminus \{0\})$
if there is a neighborhood $U \times V$ of $(x_0,\xi_0)$ such that, for any
locally constant function $\varphi$ with compact support on $U$ and any $\xi \in V$,
the Fourier transform
$\cF(\varphi u ) (\lambda \xi)$ of $\varphi u$ vanishes for $\lambda \in \Lambda\setminus C$ for some compact $C\subset K$.
 The complement of the locus of  $\Lambda$-smooth vectors in $X \times (K^n \setminus \{0\})$ is the $\Lambda$-wave front set $\WF_\Lambda(u)$ of $u$.
The goal of this paper is to study distributions and their wave front sets
 in a uniform $p$-adic context, and also in large positive characteristic, using model theoretic methods, in order to obtain uniformity results.

\medskip

 Let us now provide a more detailed description of  the content of this work.
 We start  in Section \ref{sec:non-arch} by revisiting and improving the results and constructions of Heifetz \cite{Hei85a}.
In particular,
instead of using
$K$-analytic maps as in \cite{Hei85a} we deal
with strict $C^1$ maps throughout (cf. Definition \ref{def:C^1})
Also, we work over any non-archimedean local field without any restriction on the characteristic until the end of Section \ref{sec:non-arch}.
Another important novelty is the introduction in Definition \ref{def:S'_Gamma-convergence} of a topology on the space of distributions in order to make explicit
the
continuity aspects  in Theorem \ref{thm:pull} on  pull-backs of distributions. Note that
such a topology was not considered in \cite{Hei85a},  and that if one does not specify this finer topology, the pull-back construction is not continuous and not well defined as we show in Example \ref{ex:topology}.

\medskip
  The core of this paper lies in
  Section \ref{sec:distCexpclass} were we study uniformity with respect to the local field.
  To this aim we have to use  the  field-independent descriptions provided by model theory and uniform integration, cf. \cite{CLexp}, \cite{CGH5}, \cite{CHallp}.
 More precisely, we introduce a class of distributions given by uniform, field-independent descriptions, called of $\cCexp$-class. \blue{Roughly, we require that the continuous wavelet transform is a $\cCexp$-class function in the sense of  \cite{CHallp}.} These distributions are not only uniform, they also have some geometric properties that arbitrary distributions do not share, and moreover, this class of distributions is stable under Fourier transformation and under pull-backs.
In  our study of $\cCexp$-class distributions, we make full use of \cite{CGH5} and \cite{CHallp} in  proofs: this includes use of limits in the proof of Theorem \ref{Hei2},
elimination of universal quantifiers and of the sufficiently large quantifier in the proof of Theorem \ref{Smo}, and, stability under integration in the proof of Theorem \ref{Fourier:p}.
As we show in Example \ref{ex:smooth}, one cannot expect
 the wave front sets associated to distributions of $\cCexp$-class to be definable in general.
 However, we prove in
 Theorem \ref{Smo} that
 the wave front sets associated to distributions of $\cCexp$-class are always the complement of a zero locus of a function of $\cCexp$-class. Note that
this is in sharp contrast with Theorem \ref{WFu=S} of Section \ref{sec:non-arch} which states that the wave front of an abstract distribution can be equal to any closed cone.
Our control of the wave front sets shares some similariries  in spirit with the work by Aizenbud and Drinfeld in \cite{AizDr}.
We make this connection explicit in  Section \ref{sec:WF}, where we also rephrase an open question of \cite{AizDr}, see Section \ref{sec:WF}, below Definition \ref{def:WF-hol}. 

\medskip

Finally, in  Section \ref{sec:DiscrSB} we investigate a natural question relative to the behaviour of
Schwartz-Bruhat functions under integration. What we prove is essentially that the valuative radius of balls on which the integral of
a Schwartz-Bruhat function $\varphi$ along the fibers of a morphism $f$ between projective varieties is linearly controlled by
the valuative distance to the discriminant of $f$ and the valuative radius of balls on which the function $\varphi$ is constant.

\subsection{Acknowledgements}
The authors would like to thank Julia Gordon for interesting discussions during the preparation of the paper.
The first author would like to thank Avraham Aizenbud and Dmitry Gourevitch for interesting discussions during the preparation of the paper.
The third author would like to thank
{\fontencoding{T5}\selectfont Ngô Bảo Châu} for interesting discussions leading to the results of section \ref{sec:DiscrSB}. The authors thank the referee for valuable suggestions.  

R.C. was partially supported by the European Research Council under the European Community's Seventh Framework Programme (FP7/2007-2013) with ERC Grant Agreement nr. 615722
MOTMELSUM, and thanks the Labex CEMPI  (ANR-11-LABX-0007-01). I.H. was partially supported by the SFB~878 of the Deutsche Forschungsgemeinschaft.
F.L. was partially supported by ANR-13-BS01-0006 (Valcomo), by ANR-15-CE40-0008 (D\'efig\'eo), by the European Research Council
under the European Community's Seventh Framework Programme (FP7/2007-2013)/ERC Grant Agreement nr.~246903 NMNAG and by the Institut Universitaire de France.
M.R. was partially supported  by ANR-15-CE40-0008 (D\'efig\'eo).

\section{Some additions on wave front sets in the non-archimedean case}\label{sec:non-arch}

In this section we let $K$ be a non-archimedean local field (namely either a finite field extension of $\QQ_p$ for some prime $p$ or isomorphic to $\FF_q((t))$ for some prime power $q$).
We give additions to Heifetz' constructions from \cite{Hei85a} in three ways: we use strict $C^1$ maps instead of $K$-analytic maps throughout; we work with any (non-archimedean) local field instead of only $p$-adic fields; we make the topology and continuity aspects explicit in order to define pull-backs of distributions (this is omitted in \cite{Hei85a}). From Section \ref{sec:non-arch} on, we combine this with definability conditions which will allow us to work uniformly in $K$ (usually excluding $\FF_q((t))$ of small positive characteristic however).

Let $\cO_K$ denote the valuation ring of $K$ with maximal ideal $\cM_K$ and residue field $k_K$ with $q_K$ elements and characteristic $p_K$.
Let $\abs{.}$ be the ultrametric norm on $K$ so that a uniformizer of $\cO_K$ has norm $q_K^{-1}$, and write $\ord:K\to \ZZ\cup\{+\infty\}$ for the valuation sending a uniformizer to $1$.
Let $\psi_K$ be an additive character on $K$ which is trivial on $\cM_K$ and nontrivial on $\cO_K$.

\subsection{Strict $C^1$ manifolds}
The notion of strict differentiability is rather old, but we follow \cite[Definition 7.9]{BGlockN} and \cite[Definition 3.1]{Glock2006}, and their treatments.
\begin{defn}\label{def:C^1}
	A function $f:U\subset K^n\to K^m$ with $U$ open in $K^n$ is called \emph{strictly differentiable} or \emph{strict} $C^1$ at $a\in U$ if there is a matrix $A$ in $K^{m\times n}$ such that
	$$\lim_{(x,y)\to(a,a)} \frac{f(x) - f(y) - A\cdot(x-y)}{|x-y|} =0$$
	where the limit is taken over $(x,y)\in U^2$ with $x\not=y$. Such $A$ is automatically unique and we denote it by $f'(a)$ or by $Df(a)$.
	
	The function $f$ is called \emph{strict $C^1$} if it is strict $C^1$  at each $a\in U$.
	Note that $K$-analytic maps are automatically strict $C^1$.
\end{defn}
	See \cite[Theorem A]{Glock}, \cite[Proposition 7.11]{BGlockN} and \cite[Lemma 4.4]{Glock2006} for comparisons with alternative differentiability notions. Clearly, a strict $C^1$ function is $C^1$ \cite[Lemma 3.2]{Glock2006}.\\

By the inverse and implicit function theorems for strict $C^1$  maps from \cite[Theorems 7.3, 7.4]{Glock2006}, strict $C^1$ submanifolds in $K^n$ can be defined,  with a well-defined (and unique) dimension, see \cite[Section 8]{BGlockN} and \cite[Section 2.3]{Bertram}.

\begin{defn} A \emph{strict $C^1$ chart} of $K^n$ is nothing else than $f:U\to V$ with $U\subset K^n$ and $V\subset K^n$ two open sets, $f$ a strict $C^1$ isomorphism (namely a strict $C^1$ bijection with strict $C^1$ inverse).
A nonempty subset $X\subset K^n$ is a \emph{strict $C^1$ submanifold} of $K^n$ of dimension $\ell$ for some $\ell$ with $0\leq \ell \leq n$ if for each $x\in X$ there exists a chart $f:U\to V$ of $K^n$ such that $x\in U$ and such that $f(X\cap U)$ equals $V\cap (K^\ell\times \{0\})$. 
\end{defn}

\begin{remark} \label{rem:isometry} This notion is equivalent to the following.
A nonempty subset $X$ of $K^n$, with the induced subspace topology, is a strict $C^1$ submanifold of $K^n$ of dimension $\ell$ for some $\ell\geq 0$ if, for each $x\in X$, there exist an open $U$ of $X$ containing $x$ and a coordinate projection $p:K^n\to K^\ell$ such that the restriction of $p$ to a map $p_{|U}:U\to p(U)$ is an isometry and such that $p_U^{-1}$ is strict $C^1$.
\end{remark}

By a strict $C^1$ manifold we will mean a strict $C^1$ submanifolds of $K^n$ of some dimension $\ell$ for some $n\geq \ell$. Note that, for us, strict $C^1$ submanifolds of $K^n$ are locally everywhere of the same dimension. The notion of strict $C^1$ morphisms and isomorphisms between strict $C^1$ manifolds is clear.


We use finite dimensional fiber bundles as in \cite[page 255]{BGlockN}. For example, for a strict $C^1$ submanifold  $X$ of $K^n$, the tangent bundle $TX\subset K^n\times K^n$ and the co-tangent bundle $T^*X \subset K^n\times (K^n)^*$ are well-defined. Here, $(K^n)^*$ is the dual vector space of $K^n$, and we write $x\cdot\xi$ for the evaluation of $\xi\in (K^n)^*$ in $x\in K^n$. We identify $(K^n)^*$ with $K^n$ by using the standard bases. Recall that $T^*X$ is the bundle above $X$ such that for $x\in X$, the fiber above $x$ is the dual vector space of the tangent space to $X$ at $x$. We write $T^*X\setminus \{0\}$ for the intersection of $T^*X$ with $X\times (K^n\setminus \{0\})$.
\blue{Similarly, for a strict $C^1$ submanifold  $Y$ of $X$, we can consider the normal bundle $N_Y^X$ and the co-normal bundle $CN_Y^X$ of $Y\subset X$.
Recall that $N_Y^X$ at $y\in Y$ is the quotient of the tangent space to $X$ at $y$ by the tangent space to $Y$ at $y$, and that
the co-normal bundle $CN_Y^X$ is the dual bundle of $N_Y^X$. For a smooth algebraic variety $\cX$ over $K$ and a locally closed smooth subvariety $\cY\subset \cX$ one defines the tangent bundle $T\cX$, the cotangent bundle $T^*\cX$, the normal bundle $N_Y^X$ and the co-normal bundle $CN_Y^X$ as usual, see e.g.~\cite[Section 2]{AizDr}.}
For a strict $C^1$ morphism $f:X\to Y$ between strict $C^1$ submanifolds of $K^n$, resp.~$K^m$ and $x_0\in X$, we write $Df(x_0)$ for the linear map from the tangent space $T_{x_0}X$ to $X$ at $x_0$ to the tangent space $T_{f(x_0)}Y$  to $Y$ at $f(x_0)$, and $^tDf(x_0)$ is its dual map, from the dual space of $T_{f(x_0)}Y$ to the dual of $T_{x_0}X$.



\subsection{Distributions}


For $x\in K^n$ and $r\in \ZZ$, write $B_r(x)$ for the ball $\{x'\in K^n\mid \ord (x - x') \geq r\}$, which we call a
\emph{ball of valuative radius} $r$. We will simply write $B_r$ for $B_r(0)$.
For $X$ a subset of $K^n$ (or of another space depending on the context), write $\11_{X}$ for the characteristic function of $X$.
\begin{defn}
	Let $X$ be a strict $C^1$ submanifold of $K^n$. We denote by $\mathcal C^{\infty}(X)$ the $\mathbb C$-vector space of locally constant functions on $X$ with complex values. We call functions in $\mathcal C^{\infty}(X)$ sometimes $C^\infty$-functions, which should not be confused with strict $C^1$-functions.
\end{defn}
\begin{defn}
Let $X$ be a strict $C^1$ submanifold of $K^n$, say, of dimension $\ell$. Write $\cS(X)$ for the $\CC$-vector space of
\emph{Schwartz-Bruhat functions} on $X$, namely, functions in $\mathcal C^{\infty}(X)$ with compact support.
\end{defn}

	\begin{remark} Let $\varphi$ be a nonzero Schwartz-Bruhat function in $\cS(K^n)$. Since the support of $\varphi$ is compact, there is a maximal
		integer $\alpha^{-}(\varphi)$ such that $\varphi$ is supported in the ball $B_{\alpha^{-}(\varphi)}$.
		Since $\varphi$ is locally constant and has compact support, there is a minimal integer
		$\alpha^{+}(\varphi)$ such that $\varphi$ is constant on balls of (valuative) radius $\alpha^{+}(\varphi)$.
	\end{remark}

\begin{defn}
Let $X$ be a strict $C^1$ submanifold of $K^n$.
A \emph{distribution} $u$ on $X$ is by definition a $\CC$-linear map from $\cS(X)$ to $\CC$. Write $\cS'(X)$ to denote the collection of distributions on $X$. For any Schwartz-Bruhat function $\varphi$ in $\cS(X)$, depending on the context we will denote the evaluation of $u$ on $\varphi$ by $u(\varphi)$ or $\langle u,\varphi\rangle$. The vector space $\cS'(X)$ has a structure of $\mathcal C^{\infty}(X)$-module by the following operation: For any $\mathcal C^{\infty}$ function $\phi$ on $X$ and any distribution $u$ in $\cS'(X)$, the distribution
$\phi u$ is defined by
$$\langle \phi u, \varphi \rangle = \langle u,\phi \varphi\rangle $$
for any Schwartz-Bruhat function $\varphi$ in $\cS(X)$.
\end{defn}

%

The following definition will turn out to be convenient in the definable context.

\begin{defn}\label{D-func}
Let $X$ be a strict $C^1$ submanifold of $K^n$.
The \emph{$B$-function} $D_u$ of a distribution $u$ on $X$ is the map from $X\times \ZZ$  to $\CC$ sending $(x,r)$ with $x\in X$ and $r\in\ZZ$ to $u(\11_{B_r(x)\cap X})$ if $B_r(x)\cap X$ is compact, and to zero otherwise. (Recall that $\11_A$ stands for the characteristic function of $A$.)
\end{defn}
\blue{
The $B$-function $D_u$ of a distribution $u$ on a strict $C^1$ submanifold $X\subset K^n$ may be considered (up to a scaling factor $q_K^{r\ell}$ with $\ell$ the dimension of $X$), as the $p$-adic continuous wavelet transform $W(u):X\times \ZZ\to\CC$ of $u$ with mother wavelet the characteristic function of the unit ball $B_0(0)$ around zero. More precisely, $W(u)$ is defined (by analogy to real continuous wavelet transformation) as
$$
W(u)(x,r) := q^{r\ell} D_u(x,r).
$$}
%
%

\begin{remark} Clearly $\11_{B_r(x)\cap X}$ lies in $\cS(X)$ if $B_r(x)\cap X$ is compact.
Moreover, every Schwartz-Bruhat function on $X$ is a (finite) $\CC$-linear combination of functions of the form $\11_{B_r(x)\cap X}$ with $B_r(x)\cap X$ compact, and hence, a distribution $u$ on $X$ is determined by its $B$-function.
\end{remark}

\begin{defn}\label{a-D-f}
Let $X$ be a strict $C^1$ submanifold of $K^n$ and let $D_0$ be a function on $X\times \ZZ$.
Say that $D_0$ is a \emph{function on balls} if, for any $r\in\ZZ$ and any $x,x'$ in $X$ such that $B_r(x)\cap X$
equals $B_r(x')\cap X$, one has $D_0(x,r)=D_0(x',r)$. Write $D_0(B)$ for $D_0(x,r)$ if $D_0$ is a function on balls and $B$ is the ball $B_r(x)$.
Say furthermore that $D_0$ is a $B$-function if $D_0$ is the $B$-function of a distribution on $X$. If this is the case, write $u_{D_0}$ to denote this distribution.
\end{defn}

The next (basic) lemma gives precise conditions to be a $B$-function.

\begin{lem}\label{lem:add}
Let $X$ be a strict $C^1$ submanifold of $K^n$ and let $D_0:X\times \ZZ\to\CC$ be a function. Then $D_0$ is a $B$-function if and only if $D_0$ is a function on balls such that $D_0(B)=0$ whenever $B\cap X$ is not compact, and, for any ball $B$ with $B\cap X$ compact and any finite collection of disjoint balls $B_i$ with $\bigcup_i B_i = B$ one has
\begin{equation}\label{add}
\sum_i D_0(B_i) = D_0(B).
\end{equation}
\end{lem}
\begin{proof} The direction from left to right is clear. For the other direction, suppose that $D_0$ is a function on balls and satisfies (\ref{add}) and that $D_0(B)=0$ whenever $B\cap X$ is not compact. We only need to show that the map $u$ sending a Schwartz-Bruhat function $\varphi$ with
\begin{equation}\label{add1}
\varphi = \sum_{i=1}^n c_i \11_{B_i}
\end{equation}
for some complex numbers $c_i$  and some balls $B_i$ in $X$
to
\begin{equation}\label{add2}
\sum_i c_i  D_0(B_i)
\end{equation}
is well-defined, since linearity is clear. By (\ref{add}) and by rewriting if necessary, we can suppose that all the $B_i$ are \textcolor{blue}{disjoint (and even of equal radius)}. But then, again by (\ref{add}), the value in (\ref{add2}) does not depend on the way of writing $\varphi$ as in (\ref{add1}).
\end{proof}

\subsection{Distributions on manifolds} \label{def:pushforward}

Any strict $C^1$ submanifold $X$ of $K^n$ of dimension $\ell$ comes with a natural induced $\ell$-dimensional measure, induced by the submanifold structure $X\subset K^n$. Let us denote this measure by $\mu_X$. If $U\subset X$ is open and $p:K^n\to K^\ell$ is a coordinate projection such that the restriction $p_{|U}:U\to p(U)$ is an isometry (see Remark \ref{rem:isometry}), then the measure
$p_{|U}^{*}(|d{x_{i_1}}\wedge\ldots\wedge dx_{i_\ell}|)$ on $U$ equals the restriction of $\mu_X$ to $U$, where $p$ is the projection sending $x$ to $(x_{i_1},\ldots, x_{i_\ell})$, and where $|d{x_{i_1}}\wedge\ldots\wedge dx_{i_\ell}|$ is the Haar measure on $K^\ell$ which gives measure $1$ to $\cO_K^\ell$.

Let $f:X\subset K^n\to Y\subset K^m$ be a strict $C^1$ morphism between strict $C^1$ submanifolds and let $u$ be a distribution on $X$. Suppose that the restriction of $f$ to the support of $u$ (see Definition \ref{def:support} below) is proper (proper meaning that inverse images of compact subsets are compact). Then, for $\phi\in \cS(Y)$, the composition $\phi\circ f$ lies in $\cS(X)$, and hence, we can define the push-forward $f_*(u)$  as the distribution on $Y$ sending $\phi\in \cS(Y)$ to $u(\phi \circ f)$. (In fact, $f$ being continuous instead of strict $C^1$ is enough to define $f_*(u)$.)

A distribution $u$ on a strict $C^1$ submanifold $X\subset K^n$ can be restricted to a nonempty open $U\subset X$ to a distribution denoted by $u_{|U}$ and which sends $\phi$ in $\cS(U)$ to $u(\varphi)$, where $\varphi$ is the extension by $0$ of $\phi$.

\subsection{Representation, support, singular support}

\begin{defn} Say that a distribution $u$ on a strict $C^1$ submanifold $X$ of $K^n$ is \emph{represented} by a $C^\infty$ function if there is a function $f:X\to \CC$ in $\mathcal C^\infty(X)$ so that for all $\varphi$ in $\cS(X)$ one has
$$
u(\varphi)=   \int_{x\in X}    \varphi(x) f(x) \mu_X.
$$
\end{defn}
\begin{remark}
Vice versa, a $C^\infty$ function $h:X\to\CC$ determines a distribution sending $\varphi$ in $\cS(X)$ to the integral $$
\int_{x\in X} h(x) \varphi(x) \mu_X.
$$
We thus find a map
$\mathcal C^\infty(X)\to \cS'(X)$ which is an injective linear map.
\end{remark}

\begin{defn} \label{def:support} Let $u$ be a distribution on a strict $C^1$ submanifold $X$ of $K^n$. The \emph{zero set} of $u$ is the set of points $x$ of $X$ such that there is a compact open neighborhood $U$ of $x$ where $\11_U u$ is represented by the zero function.
	The \emph{support} of $u$ is the complement in $X$ of the zero set of $u$ and is denoted by $\Supp(u)$.
\end{defn}

\begin{defn} A distribution $u$ on a strict $C^1$ submanifold $X$ of $K^n$ is \emph{smooth} at a point $x$ of $X$, if there is a compact open neighborhood
	$U$ of $x$ in $X$ such that $\11_U u$ is represented by a $C^\infty$ function. The complement in $X$ of the set of smooth points of $u$ is called the \emph{singular support} and denoted by $\Sing(u)$.
\end{defn}

\subsection{Fourier transform and oscillatory integrals}

\begin{defn}\label{def:four} For a distribution $u$ on $K^n$, the \emph{Fourier transform} $\cF(u)$ is the distribution on $K^n$ sending $\varphi$ in $\cS(K^n)$ to $u(\cF(\varphi))$, where $\cF(\varphi)$ is the Fourier transform of $\varphi$ with respect to the character $\psi_K$, namely, for $y\in K^n$,
$$
\cF(\varphi)(y) := \int_{x\in K^n} \varphi(x) \psi_K (x | y) |dx|,
$$
with $|dx|$ the normalized Haar measure on $K^n$ which gives measure $1$ to $\cO_K^n$ and where $x|y = \sum_{i=1}^n x_iy_i$.  Note that $\cF$ is a linear isomorphism from $\cS(K^n)$ to $\cS(K^n)$.
\end{defn}

For various results related to harmonic analysis on local fields, we refer to \cite{taibleson}.
For the convenience of the reader we prove the non archimedean version of the Paley-Wiener theorem \cite[Theorem 7.3.1]{Hormander83}.

\begin{thm}\label{thm:Paley-Wiener}
	Let $u$ be a distribution on $K^n$ with compact support.
	Then the distribution $\mathcal F(u)$ is represented by the following function in $\mathcal C^{\infty}(K^n)$
	$$\mathcal R_\phi : \xi \mapsto \langle u,\phi \psi_K(.\mid \xi) \rangle$$
with $\phi$ a characteristic function of a ball containing the support of $u$. Furthermore, if moreover $\mathcal F(u)$ has compact support then ${\mathcal R_\phi}$ lies in $\cS(K^n)$.
\end{thm}

\begin{proof}
As the character $\psi_K$ is trivial on $\cM_K$, for any $\xi$ in $K^n$, the function $\psi_K(.\mid \xi)$ is constant on balls of valuative radius
	$1-\ord \xi$. Let $\phi$ be the characteristic function of a ball $B_R$ containing the support of $u$.
	For any $\xi$, the function $\phi \psi_K(.\mid \xi)$ is Schwartz-Bruhat.
	By the same argument, the function $\mathcal R_\phi$ is constant on balls of valuative radius larger than $1-R$ and does not depend on such $\phi$. Indeed, if $\phi'$ is any other characteristic function of a ball containing the support of $u$, then one of these balls is included in the other and by definition of the support, $u$ vanishes on the difference $(\phi-\phi')\psi_K(.\mid \xi)$ for any $\xi$.
	
	We consider $\phi=\11_{B_R}$ as before, and we prove that the distribution $\cF(\phi u)$ is represented by the function
	$\mathcal R_\phi$, namely
	\begin{equation} \label{equality:Fourier}
	\langle \cF(\phi u),\varphi \rangle = \int_{K^n}\langle u,\phi \psi_K(.\mid \xi) \rangle\varphi(\xi)d\xi
        \end{equation}
	for any Schwartz-Bruhat function $\varphi$ in $\cS(K^n)$.
	
	By additivity, it is enough to prove the result for characteristic functions of balls, for instance for $\varphi = \11_{B_r(\xi_0)}$ with
	$r\geq 1-R$.
	By definition of the Fourier transform of distributions, we have
	$$ \langle \cF(\phi u),\varphi \rangle  =  \langle u,\phi \cF\varphi \rangle $$
	and we obtain by computation
	$$ \langle \cF(\phi u),\varphi \rangle  =  \text{vol}(B_r)\langle u,\varphi \11_{B_{1-r}}\psi_{K}(.\mid \xi_0) \rangle.$$
        By assumption on $r$, the function $\mathcal R_\phi$ is constant on the ball $B_r(\xi_0)$ which gives the equality (\ref{equality:Fourier}). The second assertion follows from the inverse Fourier transform formula.
\end{proof}

Let $X$ be an open set of $K^n$, $\phi$ be a Schwartz-Bruhat function in $\cS(X)$ and $p$ be a strict $C^1$ map
from $X\times K^r$ to $K$. For $\lambda$ in $K^{\times}$ and $\eta$ in $K^r$ we set
$$I_{\eta}(p,\phi)(\lambda) : = \int_{X}\phi(x)\psi_K(\lambda p(x,\eta)) dx.$$
The following is a non archimedean stationary phase formula, similar to \cite[Proposition 1.1]{Hei85a} but without restrictions on the characteristic of $K$. The proof is the same as for \cite[Proposition 1.1]{Hei85a}.

\begin{prop}\label{prop:stat}
Let $X\subset K^n$ and $V\subset K^r$ be two open sets.
Let $p$ be a strict $C^1$ map from $X\times V$ to $K$ and $\phi$ be in $\cS(X)$ with support $\Supp(\phi)$.
Assume that there is $\delta>0$ such that for any $(x,\eta)$ in $\Supp(\phi)\times V$ one has
$$
\abs{\grad_x p(x,\eta)}\geq \delta >0.
$$
	Suppose further that $\abs{R(x,y,\eta)}$ is bounded for $x$ and $x+y$ in $\Supp(\phi)$ and $\eta$ in $V$, where $R$ is defined by
$$
p(x+y,\eta)=p(x,\eta) + (\grad_x p(x,\eta)|y) +  (R(x,y,\eta)y|y).
$$
Then, $\lambda \mapsto I_{\eta}(p,\phi)(\lambda)$ has bounded support on $K^{\times}$, with bound independent of $\eta \in V$. Namely, there is an integer $r$ such that, for each $\eta\in V$, the support of $\lambda \mapsto I_{\eta}(p,\phi)(\lambda)$  is contained in the ball $B_r$.
\end{prop}

\subsection{Approximations}


\begin{defn}
	For $n\geq 1$, for any Schwartz-Bruhat function $\phi$ in $\cS(K^n)$ and distribution $u$ in $\cS'(K^n)$, we define the \emph{convolution  of $u$ by $\phi$} as the function
$$u*\phi : x \mapsto \langle u,\phi(x - .) \rangle.$$
\end{defn}

\begin{remark} The convolution product $u*\phi$ is a locally constant function. Indeed, as a Schwartz-Bruhat function, $\phi$ is constant on balls of radius $\alpha^{+}(\phi)$, and thus, for any $x'$ in $K^n$ with $\ord x' \geq \alpha^{+}(\phi)$ and any $x$ and $y$ in $K^n$, we have
	$$\phi((x+x')-y) = \phi(x-y)$$
 which implies
	$$(u*\phi)(x+x') = (u * \phi)(x).$$
\end{remark}

\begin{prop}[Associativity]
	For any distribution $u$ in $\cS'(K^n)$, for any Schwartz-Bruhat functions $\phi$ and $\varphi$ in $\cS(K^n)$ we have the associativity property
	$$(u*\phi)*\varphi = u*(\phi * \varphi).$$
\end{prop}
\begin{proof} By decomposition of a Schwartz-Bruhat function as a linear combination of characteristic functions of balls and by linearity of the product it is enough to prove the property in the case of $\phi = \11_{B_r(b)}$ and $\varphi = \11_{B_\eta(a)}$ with $r>\eta$.
In that case we have,
$$\begin{array}{ccl}
	(u*1_{B_r(b)})*1_{B_\eta(a)}(z) & = & \int_{K^n} \langle u,x \mapsto \11_{B_r(b)}(y-x)\rangle 1_{B_\eta(a)}(z-y)dy \\ \\
	                                & = & \int_{B_\eta(z-a)} \langle u,x \mapsto \11_{B_r(b)}(y-x)\rangle  dy \\ \\
	 			        & = & \sum_{i} \int_{B_r(y_i)} \langle u,x \mapsto \11_{B_r}(y-x-b)\rangle  dy
  \end{array}
$$
the last summation is finite and comes from the decomposition
$$B_\eta(z-a)=\bigsqcup_{i} B_r(y_i).$$
For any $i$, for any $y$ in the ball $B(y_i,r)$ we have the equivalence
$$\ord y-b-x \geq r \Leftrightarrow \ord y_i-b-x \geq r$$
and we have
$$\begin{array}{ccl}
	(u*1_{B_r(b)})*1_{B_\eta(a)}(z) & = & \sum_{i} \int_{B_r(y_i)} \langle u,x \mapsto \11_{B_r(b)}(y_i-x)\rangle  dy \\ \\
	& = & \sum_{i} \text{vol}(B_r) \langle u,x \mapsto \11_{B_r(y_i)}(x+b)\rangle  \\ \\
				        & = & \text{vol}(B_r) \langle u,x \mapsto \11_{B_\eta(z-a)}(x+b)\rangle  \\ \\
				        & = &\text{vol}(B_r) \langle u,x \mapsto \11_{B_\eta(a+b)}(z-x)\rangle
  \end{array}
$$
As $r>\eta$,  we have the following equivalences for any $z\in K^n$
$$B_r(b) \cap B_\eta(z-a) \neq \emptyset \Leftrightarrow B_r(b) \cap B_\eta(z-a) = B_r(b) \Leftrightarrow b \in B_{\eta}(z-a)
\Leftrightarrow z \in B_{\eta}(a+b).$$
We deduce the equalities
$$\11_{B_r(b)}*1_{B_\eta(a)}(z) = \int_{K^n} \11_{B_r(b)}(y)1_{B_\eta(z-a)}(y)dy = \text{vol}(B_r) \11_{B_\eta(a+b)}(z).$$
and by definition of the convolution product we obtain the equality
$$\left(u*(1_{B_r(b)}*1_{B_\eta(a)})\right)(z) = \left((u*1_{B_r(b)})*1_{B_\eta(a)}\right)(z).$$
\end{proof}

\begin{defn}\label{def:topS'}
	A sequence of distributions $u_\ell$ in $\cS'(K^n)$ for $\ell\in\NN$ is said to converge to a distribution $u$ in $\cS'(K^n)$ if and only if for any Schwartz-Bruhat function $\varphi$ in $\cS(K^n)$, the sequence of complex numbers $(\langle u_\ell,\varphi\rangle )_\ell$ converges to $\langle u,\varphi\rangle $.
\end{defn}

\begin{prop} \label{prop:uniform-convergence}
Let $(u_\ell)$ be a sequence in $\cS'(K^n)$ which converges to a distribution $u$. For any Schwartz-Bruhat function $\varphi$ in
$\cS(K^n)$, the sequence of functions $\mathcal F(\varphi u_\ell)$ converges uniformly to the function $\mathcal F(\varphi u)$ on any compact set of $K^n$.
\end{prop}

\begin{proof}
	Let $\varphi$ be a Schwartz-Bruhat function in $\cS(K^n)$. For any $\xi \in K^n$, the function
	$x\mapsto \varphi(x)\psi_K(x\mid \xi)$ is a Schwartz-Bruhat function in $\cS(K^n)$, constant on balls of radius
	$\max(\alpha^{+}(\varphi), 1-\ord \xi)$.
	Let $\ell$ be an integer. The function $\mathcal F(\varphi u_\ell)$ is $\xi\mapsto \langle u_\ell,\varphi \psi_K(.\mid \xi)\rangle $.
	In particular by definition of convergence in $\cS'$, the sequence of functions $\mathcal F(\varphi u_\ell)$
	converges pointwise to the function $\mathcal F(\varphi u)$.
	Furthermore, as the support of $\varphi$ is included in the ball $B_{\alpha^{-}(\varphi)}$,
	the function $\mathcal F(\varphi u_\ell)$ is \textcolor{blue}{constant} on balls of radius $1-\alpha^{-}(\varphi)$.
	The family of functions $\mathcal F(\varphi u_\ell)$ is then equicontinuous on $K^n$. This family is also pointwise bounded. By Ascoli-Arzela theorem, this family converges uniformly to $\mathcal F(\varphi u)$ on any compact set.
	
Alternatively, one can also prove the result without using Ascoli-Arzela theorem, as follows.
	We prove first the uniform convergence on any ball of radius larger than $1-\alpha^{-}(\varphi)$. The result for any compact set follows immediately by the Borel-Lebesgue property. Let $B$ be such a ball. If the convergence is not uniform
	then, there is $\varepsilon>0$ such that for any $M>0$ there is $m>M$ and $\xi_m \in B$ such that
	$\abs{\mathcal F(\varphi u_m)(\xi_m) - \mathcal F(\varphi u)(\xi_m)}> \varepsilon$.
	As $B$ is compact, the sequence $(\xi_m)$ has a limit point $\xi$. We can assume $(\xi_m)$ converges to $\xi$.
	By equicontinuity of the sequence $(\mathcal F(\varphi u_m))$ and by continuity of $\mathcal F(\varphi u)$ there is an integer $r$ such for any $\xi'$ in $B_r(\xi)$ and for any $m$ we have
	$$\mathcal F(\varphi u_m)(\xi') = \mathcal F(\varphi u_m)(\xi) \: \text{and} \:
          \mathcal F(\varphi u)(\xi') = \mathcal F(\varphi u)(\xi).$$
	Furthermore, by convergence of $(\xi_m)$ to $\xi$ and $(\mathcal F(\varphi u_m)(\xi))$ to
	$\mathcal F(\varphi u)(\xi)$, there is a bound $N$ such that for any $m\geq N$, we have $\xi_m \in B_r(\xi)$ implying the equalities
	$$\mathcal F(\varphi u_m)(\xi_m) = \mathcal F(\varphi u_m)(\xi) \: \text{and} \:
          \mathcal F(\varphi u)(\xi_m) = \mathcal F(\varphi u)(\xi),$$
	and the inequality
	$$\abs{\mathcal F(\varphi u_m)(\xi) - \mathcal F(\varphi u)(\xi)} \leq \epsilon.$$
        All of that implies
        $$\abs{\mathcal F(\varphi u_m)(\xi_m) - \mathcal F(\varphi u)(\xi_m)} \leq \epsilon.$$
	which is a contradiction.

\end{proof}

\begin{prop} \label{prop:approximation}
	Let $(\Phi_\ell)$ be a sequence of Schwartz-Bruhat functions of $\cS(K^n)$ supported on a neighborhood $U_\ell$ of zero such that
$$
\bigcap_\ell U_\ell = \{0\}
$$ and such that for each $\ell$
$$
\int_{K^n} \Phi_\ell(x)dx = 1.
$$
	Let $u$ be in $\cS'(K^n)$. Then the sequence $(u*\Phi_\ell)$ converges to $u$ in $\cS'(K^n)$.
\end{prop}

\begin{proof}
	For any $\ell$, we denote by $u_{\Phi_\ell}$ the locally constant function $u*\Phi_\ell$. This function defines a distribution. For any Schwartz-Bruhat function $\varphi$ in $\cS(K^n)$ we have
	$$ \langle u_{\Phi_\ell}, \varphi\rangle  =  \int_{K^n} u_{\Phi_\ell}(y) \varphi(y) dy =  (u_{\Phi_\ell}*\tilde{\varphi})(0)$$
	with $\tilde{\varphi}(y)=\varphi(-y)$.
	By associativity we conclude
	$$\langle u_{\Phi_\ell},\varphi \rangle = u*(\Phi_\ell * \tilde{\varphi})(0) = \left<u,x \mapsto \int_{K^n} \Phi_\ell(y)\varphi(x+y) dy\right>.$$
	When $\ell$ goes to infinity, the support of $\Phi_\ell$ goes to $\{0\}$, in particular there is $\ell_0$ such that for any $\ell\geq l_0$,
	the support of $\Phi_\ell$ is included in the ball $B_{\alpha^{+}(\varphi)}$ where $\varphi$ is constant and hence
	$\left<u_{\Phi_\ell},\varphi\right> = \left<u_{\Phi},\varphi\right>.$
\end{proof}

\begin{prop} \label{SB-approx-S'}
	Let $u$ be a distribution in $\cS'(K^n)$. There is a sequence of Schwartz-Bruhat functions $(u_j)$ in $\cS(K^n)$ such that the family of distributions $(u_j)$ converges to $u$.
\end{prop}
\begin{proof}
	Let $u$ be a distribution in $\cS'(K^n)$. Let $(\chi_j)$ a sequence of Schwartz-Bruhat functions in $\cS(K^n)$ such that for any compact set $C$ of $K^n$ there is an index $j_0$ such that for any $j\geq j_0$, $\chi_{j\mid C}=1$.
	Let $(\Phi_j)$ be a sequence of Schwartz-Bruhat functions in $\cS(K^n)$ such that the support of $\Phi_j$ converges to $\{0\}$
	and such that for any $j$ one has $\int \Phi_j(x)dx = 1$. For any $j$, the convolution product $(\chi_j u)*\Phi_j$ denoted by $u_j$ is a Schwartz-Bruhat functions in $\cS(K^n)$. For any $\varphi$ in $\cS(K^n)$ we have
	$$\langle u_j,\varphi\rangle = \left<u , x \to \chi_j(x)\int_{K^n} \Phi_j(y) \varphi(x+y) dy \right>.$$
Write $\tilde{\varphi}(y)=\varphi(-y)$.
	Since the support of the convolution $\Phi_j * \tilde{\varphi}$ is included in the sumset
	$\Supp  \Phi_j + \Supp  \tilde{\varphi}$, since the support of $\Phi_j$ converges to $\{0\}$ and since any compact set is included in the support of $\chi_j$ for $j$ sufficiently large, there is $j_0$ such that
	$$\Supp  \Phi_j + \Supp  \tilde{\varphi} \subset \Supp  \chi_j.$$
This implies the equality
	$$\langle u_j,\varphi\rangle  = \left<u , x \to \int_{K^n} \Phi_j(y) \varphi(x+y) dy \right>.$$
	By the previous proposition and its proof, we deduce
	$$\langle u_j,\varphi\rangle  = \langle u_{\Phi_j},\varphi\rangle  \: \to \: \langle u,\varphi\rangle .$$
\end{proof}

\subsection{Tensor product of distributions}
Let $X_1$ and $X_2$ be two open sets of $K^{n_1}$ and $K^{n_2}$. Let $u_1 \in \mathcal C(X_1)$ and $u_2 \in \mathcal C(X_2)$ be two continuous functions. The tensor product $u_1 \otimes u_2$ is defined on $X_1\times X_2$ as the function
$$u_1 \otimes u_2 (x_1,x_2)= u_1(x_1)u_2(x_2).$$
It follows from Fubini theorem that for any $\varphi_1 \in \cS(X_1)$ and $\varphi_2 \in \cS(X_2)$ we have
$$\int \int (u_1 \otimes u_2) (\varphi_1 \otimes \varphi_2)  = \left(\int u_1 \varphi_1 \right)\left(\int u_2 \varphi_2 \right).$$
\begin{thm} If $u_1 \in \cS'(X_1)$ and $u_2 \in \cS'(X_2)$, there is a unique distribution $u \in \cS'(X_1 \times X_2)$ such that
	\begin{equation} \label{equation:propriete-tensorielle}
		\langle u,\varphi_1 \otimes \varphi_2 \rangle  = \langle u_1,\varphi_1\rangle \langle u_2,\varphi_2\rangle ,
	\end{equation}
	for all $\varphi_1$ in $\cS(X_1)$ and $\varphi_2$ in $\cS(X_2)$.
	This distribution is denoted by $u_1\otimes u_2$.\\
	Furthermore, for any $\varphi \in \cS(X_1\times X_2)$, we have
	\begin{equation} \label{equation:description-produit-tensoriel}
	\begin{array}{ccl}
		\langle u,\varphi\rangle  & = & \left<u_1,x_1 \mapsto \langle u_2,x_2 \mapsto \varphi(x_1,x_2)\rangle \right> \\
		            & = & \left<u_2,x_2 \mapsto \langle u_1,x_1 \mapsto \varphi(x_1,x_2)\rangle \right>.
	\end{array}
        \end{equation}
\end{thm}
\begin{proof} A Schwartz-Bruhat function $\varphi$ in $\cS(X_1\times X_2)$ is a linear combination of characteristic functions of balls of
	$X_1\times X_2$. Futhermore, for any ball $B((a_1,a_2),r)$ of $X_1 \times X_2$ the characteristic function $\11_{B((a_1,a_2),r)}$ of the ball is the product of $\11_{B(a_1,r)}.1_{B(a_2,r)}$. This remark with the constraint \ref{equation:propriete-tensorielle} gives existence and uniqueness of the tensor product $u_1\otimes u_2$. The equalities \ref{equation:description-produit-tensoriel} follow by computation from this observation.
\end{proof}

\subsection{Wave front sets}
The wave front set $\WF(u)$ of a distribution $u$ defined on $\mathbb R^n$ is a part of $\mathbb R^n \times \big(\mathbb R^{n}\setminus \{0\}\big)$ which is conical in the second argument: for any $(x,\xi)$ in $\WF(u)$, for any positive real number $\lambda$, $(x,\lambda \xi)$ belongs to $\WF(u)$. In the non archimedean context, the analogous of the mutliplicative subgroup $\mathbb R_{>0}$ is given by Heifetz in \cite{Hei85a} as open subgroups of finite index in $K^\times$.

Let $\Lambda\subset K^\times$ be an open subgroup of finite index in $K^\times$. The intersection of such a subgroup $\Lambda$ with $\cO_K^\times$ is a finite index subgroup $\Lambda_1$ of $\cO_K^\times$. Moreover, if $a$ is an element of $\Lambda$ with minimal positive valuation among the elements of $\Lambda$, then $\Lambda = \bigcup_{i\in\ZZ} a^i \Lambda_1$. For the remainder of Section \ref{sec:non-arch} we keep $\Lambda$ fixed.

For any $n$, we consider the action of $\Lambda$ on $K^n\setminus\{0\}$ by multiplication.
This action induces an equivalence relation on $K^n\setminus\{0\}$, we denote by $S_{\Lambda}^{(n)}$ the quotient space and
we identify it with a compact subspace of $K^n$.


We define $\Lambda$-smooth points and $\Lambda$-wave fronts, generalizing Section 2, page 288 of \cite{Hei85a} (which only treats the characteristic zero case, and only on $K$-analytic manifolds). The definition uses the group $\Lambda$ and the character $\psi_K$ as fixed above.\footnote{Heifetz \cite{Hei85a} uses a different conductor than we do for $\psi_K$, but this essentially only affects explicit factors for inverse Fourier transformation and for other explicit calculations.}

First we give the definitions for distributions on an open $U\subset K^n$. Let $u$ be a distribution in $\cS'(U)$. According to the Paley-Wiener Theorem (Theorem~\ref{thm:Paley-Wiener}), $u$ is smooth in a neighborhood of a point $x$ in $U$, namely $x$ does not belong to the singular support $\Sing(u)$, if and only if there exists a Schwartz-Bruhat function $\varphi$ in $\cS(U)$ such that $\varphi(x)\neq 0$ and such that
$\mathcal F(\varphi u)$ has bounded support. In particular, if $u$ is not smooth in a neighborhood of $x$, it is natural to consider the set of critical directions in which  $\mathcal F(\varphi u)$ is not eventually vanishing. This idea underlies the following definitions.

\begin{defn}\label{Hei:smooth}
	Let $U\subset K^n$ be open and let $u$ be a distribution on $U$. Let $(x_0,\xi_0)$ be in $U\times (K^n\setminus\{0\})$.  Say that $u$ is \emph{$\Lambda$-micro-locally smooth} or \emph{$\Lambda$-smooth} at $(x_0,\xi_0)$ if there are open neighborhoods $U_0$ of $x_0$ and $V_0$ of $\xi_0$ such that for any Schwartz-Bruhat function $\varphi$ with support contained in $U_0$ there is an integer $N$ (which may depend on $U_0$, $V_0$ and $\varphi$) such that for all $\lambda\in\Lambda$ with $\ord \lambda < N$ one has for all $\xi$ in $V_0$ that
\begin{equation}\label{Hei:p}
\cF(\varphi u) (\lambda\xi) = 0.
\end{equation}
\end{defn}

%

\begin{defn}\label{defn:wf}
Let $u$ be a distribution on an open $U\subset K^n$.
The complement in $U\times (K^n\setminus \{0\})$ of the set of $(x_0,\xi_0)$ at which $u$ is $\Lambda$-smooth, is called the
\emph{$\Lambda$-wave front set} of $u$, and is denoted by $\WF_\Lambda(u)$.
\end{defn}

By compactness and the definitions, the wave front set lies above the singular support, as follows. We leave the details of the proof to the reader.
\begin{prop}
	Let $u$ be a distribution on an open $U\subset K^n$. Let $\pi$ be the projection from  $K^n \times K^n\setminus \{0\}$ to $K^n$. Then the projection
	$\pi(\WF_\Lambda(u))$ equals the singular support $\Sing (u)$.
\end{prop}

\begin{remark}
	The local nature of wave front sets 
is clear from the definition (as in \cite[Proposition 2.1]{Hei85a}): Let $V\subset U\subset K^n$ be open and let $u$ be in $\cS'(U)$. Then one has
	\begin{equation}\label{localWF}
		\WF_\Lambda(u_{|V}) = \WF_\Lambda(u) \cap (V\times K^n).
	\end{equation}
\end{remark}

\begin{remark}\label{rem:WFcharts} Let $f:U\to U_1$ be a strict $C^1$ isomophism between two open sets of $K^n$ and $u$ be a distribution in $\cS'(U)$.
	A point
$$(x_0,\xi_0)\mbox{ in }T^*U\setminus\{0\}
$$
is $\Lambda$-microlocally smooth for $u$ if and only if the point
$$
\left(f(x_0),(\:^{t}Df(x_0))^{-1}(\xi_0)\right)
$$
is $\Lambda$-microlocally smooth for $f_*u$. Here, $\:^{t}Df(x_0)$ stands for the transpose of $Df(x_0)$, or, in other words, for the dual linear map of $Df(x_0)$ by the identifications made between $K^n$ and $(K^n)^*$.
\end{remark}
\begin{proof}[Proof of Remark \ref{rem:WFcharts}]
The remark follows from Proposition \ref{prop:stat} using the function $p$ defined by
$$
p(x,\eta)=(f(x) \mid \eta),
$$
similar as in the proof of \cite[Proposition 2.3]{Hei85a} and \cite[Proposition 1.3.2]{Duist}. Let us repeat the argument for the convenience of the reader. If $u$ is $\Lambda$-microlocally smooth at $(x_0,\xi_0)$ and if we put $\eta_0= (\:^{t}Df(x_0))^{-1}(\xi_0)$, then, by Proposition \ref{prop:stat} and by local Taylor expansions (of degree $1$) of strict $C^1$ functions from \cite[Theorem 5.1 and Proposition 5.3]{BGlockN}, there are neighborhoods $U_{x_0}$ of $x_0$ and
$\check{U}_{\eta_0}$ of $\eta_0$ such that for any $\varphi$ in $\cS(U_{x_0})$, there exists an $N_\varphi>0$, such that for any $\lambda$ in $\Lambda$ with, $\ord \lambda \leq N_\varphi$ we have for any $\eta$ in $\check{U}_{\eta_0}$
$$
\langle u,\varphi \psi_K(\lambda(f(.)\mid \eta)\rangle  = 0.$$
	Using the definition of the push-forward by $f$, this means
	$$\langle f_*u, (\varphi\circ f^{-1})\psi_K(\lambda(.\mid \eta)\rangle  = 0.$$
	We deduce that $(f(x_0),\eta_0)$ is a $\Lambda$-microlocally smooth point of $f_*u$. The other implication follows similarly using $f^{-1}$ instead of $f$.
\end{proof}

For $X$ a strict $C^1$ submanifold of $K^n$, recall that $T^*X$ stands for the co-tangent bundle of $X$, and  $T^*X\setminus\{0\}$ for the set of $(x,\xi)$ in $T^*X$ with $x\in X$ and $\xi\not=0$.
Using the previous two remarks, and again the Taylor approximation formula (of degree $1$) for strict $C^1$ functions from \cite[Theorem 5.1 and Proposition 5.3]{BGlockN}
we can give now the `coordinate free' definition of wave front sets for distributions on strict $C^1$ submanifolds.

\begin{def-lem}[Wave front sets]\label{Hei:smoothvar}
	Let $X\subset K^n$ be a strict $C^1$ submanifold of dimension $\ell$ and let $u$ be a distribution on $X$. Let $(x_0,\xi_0)$ be in $T^*X\setminus\{0\}$. Say that $u$ is \emph{$\Lambda$-micro-locally smooth} or \emph{$\Lambda$-smooth} at $(x_0,\xi_0)$ if there are an open $U\subset X$ containing $x_0$ and a strict $C^1$ isomorphism $f:U\to U_1\subset K^\ell$, such that the distribution $f_*(u_{|U})$ on $U_1$
is  $\Lambda$-smooth at $(f(x_0),(^tDf(x_0))^{-1}\xi_0)$.
Moreover, one has that $u$ is $\Lambda$-smooth at $(x_0,\xi_0)$ if and only if for all open set $U\subset X$ containing $x_0$ and all strict $C^1$ isomorphism $f:U\to U_1\subset K^\ell$, the distribution $f_*(u_{|U})$ on $U_1$
is  $\Lambda$-smooth at $(f(x_0),(^tDf(x_0))^{-1}\xi_0)$.
The complement in $T^*X\setminus\{0\}$ of the set of $(x_0,\xi_0)$ at which $u$ is $\Lambda$-smooth, is called the
\emph{$\Lambda$-wave front set} of $u$, and is denoted by $\WF_\Lambda(u)$.
\end{def-lem}

\begin{defn}
	A subset $\Gamma$ of $T^*X\setminus\{0\}$ is called a \emph{$\Lambda$-cone} in $T^*X\setminus\{0\}$ if and only if, for each $\lambda\in \Lambda$ and each $(x,\xi)$ in $\Gamma$, the point $(x,\lambda\xi_1,\ldots,\lambda\xi_n)$, denoted by $(x,\xi)$, lies in $\Gamma$.
\end{defn}

\begin{remark}
	By definition, the wave front set of a distribution is an example of a $\Lambda$-cone.
\end{remark}

Wave front sets can be rather general sets, as shown by the following result, analogous to \cite[Thm 8.1.4]{Hormander83}.

\begin{thm}\label{WFu=S}
	Let $\Lambda$ be a finite index subgroup of $(K^{*},\times)$.
	Let $d \geq 1$ be an integer. Let $S$ be a closed subset of $K^d \times (K^{d} \setminus \{0\})$ which is $\Lambda$-conical in the second factor meaning that $S$, when seen as subset of $T^*K^{d} \setminus \{0\})$, is a $\Lambda$-cone. Then there is a distribution $u$ in $\cS'(K^d)$ such that $\WF_{\Lambda}(u)=S$.
\end{thm}
\begin{proof} We adapt the proof of \cite[Thm 8.1.4]{Hormander83} to the present context.
Let $\lambda$ be an element of $\Lambda$ with minimal positive valuation.
	      Let $(x_k,\theta_k)_k$ be a sequence in $S$ indexed by integers $k>0$, with $\abs{\theta_k}_K = 1$ for each $k$ and such that any point $(x,\theta)$ of $S$ with $|\theta|=1$ is a limit of a subsequence of $(x_k,\theta_k)_k$.
	      We consider the function $u$ on $K^d$ defined for $x$ in $K^d$ by
	      $$
u(x) =
	      \sum_{k >0} \abs{\lambda^{k}} \11_{\cO_K^d}(\lambda^{-k}(x-x_k))
	      \psi_K(x\mid -\lambda^{-3k}\theta_k),
$$
where $\11_{\cO_K^d}$ is the characteristic function of $\cO_K^d$.
Clearly $u$ is locally integrable on $K^d$. Moreover, $u$ is locally constant outside the projection $\pi_x(S)$, where $\pi_x$ is the projection to $K^d$.
Indeed, for $x_0\notin \pi_x(S)$, as $S$ is closed,  $|\theta_k|=1$ for each $k>0$ and by compactness, there is a neighborhood $U_{x_0}$ such that at most finitely many $x_k$ belong to $U_{x_0}$, and thus, $u_{\mid U_{x_0}}$ is a finite sum of locally constant functions. In particular, the singular support of $u$ is contained in $\pi_x(S)$.

\par
	We now prove the inclusion $\WF_{\Lambda}(u) \subset S$. 
	Fix $(x_0,\xi_0) \notin S$ with $x_0\in \pi_x(S)$. Then there is a neighborhood $U$ of $x_0$ and a $\Lambda$-conical neighborhood $V$ of
	$\xi_0$ such that
	\begin{equation} \label{UVS}
		(U \times V)\cap S = \emptyset.
	\end{equation}
 We may suppose that $U$ is closed and that $\overline V = V\cup \{0\}$.
	We write $u$ as $u=u_1+u_2$ where $u_1$ is the sum of terms for $k>0$ with $x_k\notin U$ and $u_2$ is the sum of terms for $k>0$ with
	$x_k \in U$.
\textcolor{blue}{Then, it follows from its definition that \blue{the restriction to $U$ of } $u_1$ is a finite sum of locally constant functions on $U$.}
For the Fourier transform of $u_2$ we find
		\begin{equation} \label{Fourieru} 	
		\mathcal F(u_2)(\xi) = \sum_{k>0,\ x_k \in U} \abs{\lambda^{(d+2)k} }
		 \11_{B_1} \left( \lambda^k \xi-\lambda^{-2k}\theta_k  \right) \psi_K(x_k \mid (\xi-\lambda^{-3k}\theta_k)),
	\end{equation}
by integrating term by term and by noting that $\mathcal F(\11_{\cO_K^d})= \11_{B_1}$ with $B_1\subset K^d$ the ball around $0$  of valuative radius $1$.
 	As $(x_k,\theta_k)$ lies in $S$ for each $k>0$ we have that $\theta_k \notin V$ for each $k>0$.
	Hence, there is a constant $c>0$ such that for any $\xi \in V$ and $\theta \notin V$ we have
	$$\abs{\xi - \theta}\geq c  \abs{ (\xi,\theta) },$$
the norm of a tuple being the maximum of the norms of the entries.
	(The existence of $c$ is clear for $(\xi,\theta)$ with $\abs{(\xi,\theta)} = 1$ and thus for general $(\xi,\theta)$ by scaling.)
	Thus, for any $\xi$ in $V$ and any $k>0$ one has
		\begin{equation} \label{xithetak} 	
    \abs{ \lambda^k \xi-\lambda^{-2k}\theta_k  } \geq c   \max(\abs{\lambda^k \xi}, \abs{-\lambda^{-2k}}) \geq  c\abs{\xi}^{2/3} .
	\end{equation}
By the presence of $\11_{B_1}$ in (\ref{Fourieru}) and by (\ref{xithetak}), the restriction of $\mathcal F u_2$ to $V$ has bounded support, and hence, $(x_0,\xi_0)\notin \WF_\Lambda(u)$. We have showed that $\WF_{\Lambda}(u) \subset S$.

\par
 Let us now prove the inclusion $S\subset \WF_{\Lambda}(u)$. Let $(x_0,\xi_0)$ be in $S$. Let us write $\chi$ for $\11_{B_0(x_0)}$.
For any $k>0$ such that $x_k\not \in B_0(x_0)$ and for any $x$ in the ball $B_0(x_0)$, one has
$$
\11_{\cO_K^d}(\lambda^{-k}(x-x_k))=0.
$$
Hence, the Fourier transform $\mathcal F(\chi u)$ has the same form as
$\mathcal F(u)$ in (\ref{Fourieru}) but with the condition $x_k \in U$ on $k>0$ replaced by the condition $x_k\in B_0(x_0)$. We obtain for any $k>0$ the equality
	\begin{multline*} \mathcal F(\chi u)(\lambda^{-3k}\theta_k) = \\
\sum_{j>0,\ x_j\in B_0(x_0)} \abs{\lambda^{(d+2)j} }
			 \11_{\cO_K^d}\left( \lambda^{-3k+j}\theta_k - \lambda^{-2j}\theta_j \right) \psi_K\left(x_j \mid \lambda^{-3k}\theta_k - \lambda^{-3j}\theta_j \right).
		\end{multline*}
Note that for any $j>0$ with $j\neq k$ one has
	$$
\abs{  \lambda^{-3k+j}\theta_k - \lambda^{-2j}\theta_j }  \geq \abs{\lambda}^{-k},
$$
which implies for all $k>0$ with
$x_k \in B_0(x_0)$ that
$$
\mathcal F(\chi u)(\lambda^{-3k}\theta_k) = \abs{\lambda^{(d+2)k} } \neq 0.
$$
In particular, $(x_0,\xi_0)\notin \WF_\Lambda(u)$, which finishes the proof.
\end{proof}

For any distribution $u$ on a strict $C^1$-manifold $X$, we write $\WF_\Lambda^{0} (u)$ for the union of $\WF_\Lambda (u)$ with $X\times \{0\}$ in $T^*X$. With this notation, we have the following relation for tensor products.

\begin{thm}[Wave front sets and tensor products] \label{thm:WFtenseur}
Let $X_1$ and $X_2$ be two open sets of $K^{n_1}$ and $K^{n_2}$ respectively. Let $u_1$ be in $\cS'(X_1)$ and $u_2$ be in $\cS'(X_2)$.
Then we have the inclusion
\begin{equation} \label{equation:inclusion-front-onde-tenseur}
\begin{array}{ccl}
\WF_\Lambda(u_1\otimes u_2) & \subset &   \WF_\Lambda^{0} (u_1) \times \WF_\Lambda^{0} (u_2).
\end{array}
\end{equation}
\end{thm}
\begin{proof}
Denote by $\mathcal E$ the right-hand side term of the inclusion (\ref{equation:inclusion-front-onde-tenseur}).
	Let $\mathcal P$ be a point $((a_1,\eta_1),(a_2,\eta_2))$  in $T^*X_1 \times T^*X_2$, with $(\eta_1,\eta_2)\not=0$.
	Assume $\mathcal P$ does not belong to $\mathcal E$.
	In particular $(a_1,\eta_1)$ does not belong to $\WF_\Lambda(u_1)$
	or $(a_2,\eta_2)$ does not belong to $\WF_\Lambda(u_2)$.
	Assume for instance $(a_1,\eta_1) \notin \WF_\Lambda(u_1)$. Then, by the definition of $\Lambda$-micro-local smoothness, there are an open ball $U_{a_1}$ centered at $a_1$ and an open ball $\check U_{\eta_1}$ centered at $\eta_1$ such that for any Schwartz-Bruhat function $\varphi$ in $\cS(U_{a_1})$ there is an integer $N_{\varphi}>0$ such that for all $\lambda$ in $\Lambda$ with $\ord \lambda < N_\varphi$ and $\xi_1$ in $\check U_{\eta_1}$ we have
	$$\mathcal F(\varphi u_1)(\lambda \xi_1)=0.$$
Consider the open sets $\Omega = U_{a_1}\times X_2$ and $\check \Omega = \check U_{\eta_1} \times K^{n_2}$ and choose $\varphi \in \cS(\Omega)$. This Schwartz-Bruhat function is a $\CC$-linear combination of characteristic functions of balls of $\Omega$ which themselves are products of characteristic functions of balls in $U_{a_1}$ and in $X_2$, and thus, $\varphi$ can be written such that for $x_1\in X_1$ and $x_2\in X_2$
$$
\varphi(x_1,x_2) = \sum_{i=1}^{\ell} c_{i}\11_{B_{i,1}}(x_1)\11_{B_{i,2}}(x_2),
$$
	with $B_{i,1}$ a ball in $U_{a_1}$, $B_{i,2}$ a ball in $X_2$ and $c_{i}$ a complex number for each $i$.
	Denote by $N_\varphi$ the minimum $\min N_{\11_{B_{i,1}}}$ over $i$.
	Using (\ref{equation:description-produit-tensoriel}), we conclude that for any $\lambda$ in $\Lambda$ with $\ord \lambda \leq N_\varphi$ and $(\xi_1,\xi_2)$ in $\check \Omega$ we have
	\begin{equation} \label{equation:WFtenseur}
			\left<u_1\otimes u_2,(x_1,x_2) \mapsto \varphi(x_1,x_2)\psi_K((x_1,x_2) \mid \lambda (\xi_1,\xi_2))\right> =0,
			        \end{equation}
	which shows that the point $((a_1,\eta_1),(a_2,\eta_2))$ does not belong to $\WF_{u_1 \otimes u_2}$ as desired.
	\end{proof}
	
	\subsection{Pull-backs and push-forwards of distributions}\label{sec:sequences}

Let $X$ be a strict $C^1$ submanifold of $K^n$. On the set $\cS'(X)$ of all distributions on $X$, one has the topology such that for distributions $u$ and $u_j$ on $X$, one has $u_j\to u$ if and only if for all Schwartz-Bruhat functions $\varphi$ on $X$ one has $u_j(\varphi)\to u(\varphi)$ in $\CC$. We call this convergence of $(u_j)_j$ in the $\cS'$-sense. We introduce a refined topology on certain subsets  $\cS'_{\Gamma,\Lambda}(X)$  of $\cS'(X)$ by using closed $\Lambda$-cones $\Gamma$ in $T^* X$.


\begin{defn} \label{def:S'_Gamma-convergence}
	Let $\Gamma$ be a closed $\Lambda$-cone in
	$T^*X \setminus \{0\}$. We denote by $\cS'_{\Gamma,\Lambda}(X)$ the set of distributions $u$ in $\cS'(X)$ such that
	the wave front set $\WF_{\Lambda}(u)$ is included in $\Gamma$. A sequence $(u_j)_j$ with $u_j$ in $\cS'_{\Gamma,\Lambda}(X)$ is said to converge in the
	$\cS'_{\Gamma,\Lambda}$-sense to a distribution $u$ in $\cS'(X)$ if and only if
	\begin{enumerate}
		\item $(u_j)_j$ converges to $u$ in the $\cS'$-sense,
		\item \label{def:S'_Gamma-convergence-Fourier} for each $(x_0,\eta_0)$  in $T^*X\setminus \{0\}$ with $(x_0,\eta_0)\not\in \Gamma$, there is an open neighborhood $U_{x_0}$ of $x_0$ and an open neighborhood $U_{\eta_0}$ of $\eta_0$ such that for any $\varphi \in \cS(U_{x_0})$ there is an integer $N\in \mathbb Z$ such that
$$
\mathcal F(\varphi u_j)(\lambda \eta)=0
$$
for all $j$, all $\eta \in U_{\eta_0}$ and for all $\lambda \in \Lambda$ with $\ord \lambda <N$.
	\end{enumerate}
\end{defn}

\begin{prop}\label{SB-approx-S':bis}
Let $X$ be a strict $C^1$ submanifold of $K^n$ and let $u$ be a distribution in $\cS'(X)$. Then there exists a sequence of Schwartz-Bruhat functions $(u_j)_j$ 
which converges to $u$ in the $\cS'_{\Gamma,\Lambda}$-sense if and only if $u$ lies in $\cS'_{\Gamma,\Lambda}(X)$.
\end{prop}
\begin{proof}
Since we can work locally, it is enough to treat the case that $X=K^n$.
As in the proof of proposition \ref{SB-approx-S'} and with its notation for $\chi_j$ and $\Phi_j$, we consider a sequence $(u_j)_j$  of Schwartz-Bruhat functions in $\cS(K^n)$, with $u_j = (\chi_j u)*\Phi_j$, which converges to $u$ in the $\cS'$-sense.
Each function $u_j$ defines a distribution with an empty wave front set, hence $u_j$ lies in $\cS'_{\Gamma,\Lambda}(K^n)$.

\par
Let us first suppose that $u$ lies in $\cS'_{\Gamma,\Lambda}$ and prove the $\cS'_{\Gamma,\Lambda}$-convergence of the sequence to $u$.
Let $(x_0,\eta_0)$ be a point in $T^* K^n\setminus \{0\}$ not lying in $\Gamma$. In particular, by the assumptions on $u$, the point $(x_0,\eta_0)$ is $\Lambda$- micro-locally smooth for $u$, and thus, there is a ball $B(x_0,r)$ and an open neighborhood $U_{\eta_0}$ of $\eta_0$ such that for any
	$\varphi \in S(B(x_0,r))$ there is a constant $N$ such that $\mathcal F(\varphi u)(\lambda \eta)=0$  for all $\eta \in U_{\eta_0}$ and all $\lambda \in \Lambda$ with $\ord \lambda \leq N$.
	We denote by $U_{x_0}$ the ball $B_{2r}(x_0)$. For $\varphi \in \cS(U_{x_0})$ it is sufficient to prove the existence of an index $j_0$ such that $\mathcal F(\varphi u_j)(\lambda \eta)=0$  for any $j\geq j_0$, $\eta \in U_{\eta_0}$ and  $\lambda \in \Lambda$ with $\ord \lambda \leq N$. To this end, note that
	\begin{multline*}
		\langle \varphi u_j, \psi_K(\cdot \mid \lambda \eta)\rangle  = \langle u_j, \varphi \psi_K(\cdot \mid \lambda  \eta) \rangle
		= \langle  u ,\chi_j.(\Phi_j*  (\varphi \psi_K(\cdot \mid - \lambda  \eta )))\rangle  \\
		= \left < u ,x \to \chi_j(x).\left[\int_{y}\varphi(x+y)\psi_K(y\mid \lambda \eta) \Phi_j(y)dy\right]\psi_K(x\mid \lambda \eta) \right>
        \end{multline*}
and that there is an index $j_0$ such that for $j$ with $j\geq j_0$ one has
	$$\Supp  \chi_j \cap (\Supp  \Phi_j + \Supp  \varphi) \: = \: (\Supp  \Phi_j + \Supp  \varphi) \subset B(x_0,r).
$$
The function
$$
x\to \chi_j(x) \int_{y}\varphi(x+y)\psi_K(y\mid \lambda \eta) \Phi_j(y)|dy|
$$
is clearly a Schwartz-Bruhat function,
with support in the ball $B(x_0,r)$. By the microlocal smoothness we obtain
	$$\langle  u ,\chi_j.(\Phi_j* {\varphi \psi_K(\cdot\mid -\lambda \eta)})\rangle  = 0 $$
	for all $j>j_0$, all $\lambda$ with $\ord \lambda \leq N$ and all $\eta \in U_{\eta_0}$. This proves the $\cS'_{\Gamma,\Lambda}$-convergence.
The remaining implication follows easily from (2) and (1), by definition of $\cS'$-convergence.
\end{proof}


Let $f:X\subset K^n\to Y\subset K^m$ be a strict $C^1$ morphism of strict $C^1$ submanifolds.
Composition with $f$ gives a pullback morphism $$f^*:\mathcal C^\infty(Y)\to \mathcal C^\infty(X):h\mapsto h\circ f.$$

For $\Gamma$ a $\Lambda$-cone in $T^*Y\setminus\{0\}$, define the cone $f^*(\Gamma)$ by
$$
f^*\Gamma = \{ (x,\xi) \in T^*X\setminus\{0\} \mid \exists (y,\eta)\in \Gamma\, \mbox{ with } f(x)=y \mbox{ and } \xi = {}^tDf(x)\eta\}.
$$
Likewise, for $\Delta$ a $\Lambda$-cone in $T^*X\setminus\{0\}$, define the cone $f_*(\Delta)$ by
$$
f_*\Delta = \{ (y,\eta) \in T^*Y\setminus\{0\} \mid \exists (x,\xi)\in \Delta\, \mbox{ with } f(x)=y \mbox{ and } \xi = {}^tDf(x)\eta) \}.
$$
Finally, define the $\Lambda$-cone $N_f$ by
$$
N_f = \{ (y , \eta) \in T^*Y\setminus\{0\} \mid \exists x\in X \mbox{ with $f(x)=y$ and } 0={}^tDf(x)\eta \}.
$$

\begin{thm}[Pull-backs of distributions]\label{thm:pull}
Let $f:X\subset K^n\to Y\subset K^m$ be a strict $C^1$ morphism of strict $C^1$ submanifolds.  Let $\Gamma$ be a closed $\Lambda$-cone in $T^*Y\setminus\{0\}$ such that
$$
N_f\cap \Gamma = \emptyset.
$$
Then the pullback
$$
f^*:\mathcal C^\infty(Y)\to \mathcal C^\infty(X),\: h \mapsto h \circ f
$$
has a unique continuous extension

$$
f^*:\cS'_{\Gamma,\Lambda}(Y) \to \cS'(X),
$$
for the topology from Definition \ref{def:S'_Gamma-convergence} on $\cS'_{\Gamma,\Lambda}(Y) $ and from Definition \ref{def:topS'} on $\cS'(X)$.
Moreover, 
$\Supp (f^* u) \subset f^{-1}(\Supp u)$ and $f^*(u)$ lies in $\cS'_{f^*\Gamma,\Lambda}(X)$ for each $u\in \cS'_{\Gamma,\Lambda}(Y)$.
Furthermore, $f^*$ induces a continuous map
$$
\cS'_{\Gamma,\Lambda}(Y) \to \cS'_{f^*\Gamma,\Lambda}(X),
$$
with both topologies from Definition \ref{def:S'_Gamma-convergence}.
\end{thm}
By Taking $\Gamma = \WF_\Lambda(u)$ in Theorem \ref{thm:pull} one finds that $f^*(u)$ is defined when $\WF_\Lambda(u)$ is disjoint from $N_f$, with furthermore
$$
\WF_\Lambda(f^*u) \subset f^* (\WF_\Lambda(u)).
$$

\begin{proof}[Proof of Theorem \ref{thm:pull}]
Uniqueness follows from the continuity of $f^*$ and Proposition \ref{SB-approx-S':bis}.
It is sufficient to prove the existence in the case that $X=K^m$ and $Y=K^n$ by working locally and with charts.
We adapt the proofs of \cite[Theorem 2.8]{Hei85a} and \cite[8.2.4]{Hormander83}, with extra focus on continuity in comparison to \cite[Theorem 2.8]{Hei85a}.
Using the density of $\cS(Y)$ in $\cS'_{\Gamma, \Lambda}(Y)$, we define the pullback morphism $f^*$ from $\cS'_{\Gamma,\Lambda}(Y)$ to $\cS'(X)$. More precisely, for a distribution $u$ in $\cS'_{\Gamma,\Lambda}(Y)$, for any $x_0$ in $X$ we define a convenient neighborhood $X_0$ of $x_0$ and $f^*u$ on $\cS(X_0)$. The construction $f^*u$ extends naturally to $\cS(X)$ by linearity and compactness of the support of a Schwartz-Bruhat function. We will finally need to show that this is well-defined, continuous and goes to
$\cS'_{f^*\Gamma,\Lambda}(X)$.

Let $x_0 \in X$, $y_0= f(x_0)$. If, $y_0$ does not belong to $\pi_Y(\Gamma)$, then as $\WF_\Lambda(u)$ is included in $\Gamma$, $y_0$ is a smooth point, and locally around $y_0$, $u$ is a $C^{\infty}$ function, and the pull-back
$f^*u$ is already defined as $u \circ f$ around $y_0$. Assume now that $y_0$ belongs to $\pi_Y(\Gamma)$, then
$\Gamma_{y_0} = \{\eta \mid (y_0,\eta) \in \Gamma\}$ is non empty.

In the following steps we construct some neighborhoods $V$ of $\Gamma_{y_0}$, $Y_0$ of $y_0$ and $X_0$ of $x_0$, and we start the construction of $\langle f^*u,\chi\rangle $ for any $\chi$ in $\cS(X_0)$.
Recall that the action of $\Lambda$ on $K^n\setminus\{0\}$ induces an equivalence relation on $K^n\setminus\{0\}$, we denote by $S_{\Lambda}^{(n)}$ the quotient space and
we identify it to a compact subspace of $K^n$.

\emph{Step 1}. As $\Gamma_{y_0} \cap N_{f,y_0}$ is empty, by continuity on the compact set
$\Gamma_{y_0}\cap S_{\Lambda}^{(n)}$, the map $\:^{t}Df(x_0)$ admits a minimum $\delta>0$. By compactness, there is an open-closed conic neighborhood $V$ of $\Gamma_{y_0}$ in $K^n \setminus \{0\}$ with $^{t}Df(x_0)\eta \neq 0$ for any $\eta \in V$. For instance we can consider $V$ as
$$V = \Lambda \left(\bigcup_{\eta \in \Gamma_{y_0}\cap S_{\Lambda}^{(n)}} B_\eta \right).$$
where $(B_\eta)$ is a finite covering of $\Gamma_{y_0}\cap S_{\Lambda}^{(n)}$ by balls $B_\eta$ centered at $\eta$, such that for any $\eta$ in $\Gamma_{y_0}\cap S_{\Lambda}^{(n)}$, for any $\eta'\in B_\eta$, $\abs{ ^{t}Df(x_0)\eta'}>\delta$.

\emph{Step 2}. As $\Gamma$ is closed, there is a compact neighborhood $Y_0$ of $y_0$ such that $V$ is a neighborhood of $\Gamma_y$ for every $y \in Y_0 \cap \pi_Y(\Gamma)$. Else, considering a familly of balls centered at $y_0$ with valuative radius
$r_n \to +\infty$, there is a sequence $(y_\ell,\eta_\ell)$ with $\eta_\ell$ in $\Gamma_{y_\ell}\cap S_{\Lambda}^{(n)} \setminus V$. By compactness of $S_{\Lambda}^{(n)}$, we can extract a subsequence converging to a point $(y_0,\eta)$. As $V$ is open and $\Gamma$ is closed , $(y_0,\eta)$ belongs to $\Gamma_{y_0}\setminus V$ which shows that $\Gamma_{y_0} \not\subset V$, a contradiction.

\emph{Step 3}. There is a compact neighborhood $X_0$ of $x_0$ with $f(X_0)$ in the interior of $Y_0$ and $^{t}Df(x)\eta \neq 0$ for any
$x \in X_0$ and $\eta \in V$.
Indeed, there is $\delta>0$ such that for any $\eta_0$ in $V\cap S_{\Lambda}^{(n)}$, $\abs{\:^{t}Df(x_0)\eta_0}>\delta$.
By continuity of $^{t}Df(.).$, there are neighborhoods $U_{x_0,\eta_0}$ of $x_0$ and $\check{U}_{x_0,\eta_0}$ of $\eta_0$ such that for any $x \in U_{x_0,\eta_0}$ and $\eta \in \hat{U}_{x_0,\eta_0}$ we have $\abs{\:^{t}Df(x)\eta}>\delta$.
By compactness of $V\cap S_{\Lambda}^{(n)}$ and Borel-Lebesgue property, there is a finite covering of
$V\cap S_{\Lambda}^{(n)}$ by open sets $(\hat{U}_{x_0,\eta_i})$ and we consider $U_{x_0} = \cap_{i} U_{x_0,\eta_i}$.
We consider $X_0$ as an open compact neighborhood contained in $f^{-1}(Y_0)\cap U_{x_0}$.

\emph{Step 4}. By proposition \ref{SB-approx-S':bis}, let $(u_\ell)$ be a sequence of $\cS(Y)$ which converges to $u$ for the $\cS'_{\Gamma, \Lambda}$ convergence.
We can assume $Y_0$ sufficiently small such that for any $\varphi \in \cS(Y_0)$, there is an integer $j_\varphi$ such that, for any $j\geq j_\varphi$, $\mathcal F(\varphi u_j)1_{V^c}$ has a \textcolor{blue}{compact support which is moreover} independent from $j$ where
$V^c$ is the complement of $V$ in $K^{n} \setminus \{0\}$. This follows from the compactness of $V^c \cap S_{\Lambda}^{(n)}$ and the definition of the $\cS'_{\Gamma, \Lambda}$ convergence.

Using, $x_0$, $y_0$, $V$, $X_0$, $Y_0$ and $(u_\ell)$ the sequence of $\cS(Y)$ defined above, we give now the construction of $\langle f^*u,\chi\rangle $ for any $\chi$ in $\cS(X_0)$.
We fix also $\varphi$ be in $\cS(Y_0)$ equal to $1$ on $f(X_0)$ hence for any $\ell$, $f^{*}(\varphi u_\ell)$ is equal to
$f^{*} u_\ell$ and below nothing will depend on $\varphi$.
Let $\chi$ be a Schwartz-Bruhat function in $\cS(X_0)$. By the Fourier's inversion formula, we have for any $\ell$,
$$\langle f^{*}(\varphi u_\ell),\chi\rangle  = \int_{\eta \in K^n} \mathcal F(\varphi u_\ell)(\eta) I_{\chi}(\eta) d\eta $$
with
\begin{equation} \label{formuleIchi}
I_{\chi}(\eta) = \text{vol}(B_1)\int_{x \in X}\chi(x)\psi(-f(x)\mid \eta) dx.
\end{equation}
%
By construction of $X_0$, compactness of $V\cap S_{\Lambda}^{(n)}$ and the stationary phase formula, the function $I_{\chi}1_{V}$ has \blue{a compact support}.
By proposition \ref{prop:uniform-convergence}, the sequence $\mathcal F(\varphi u_\ell)$ converges uniformly to
$\mathcal F(\varphi u)$, then we deduce the convergence
$$\int_{V} \mathcal F(\varphi u_\ell)(\eta) I_{\chi}(\eta) d\eta \to
\int_{V} \mathcal F(\varphi u)(\eta) I_{\chi}(\eta) d\eta.$$
%
By \emph{step 4}, $\mathcal F(\varphi u_\ell)$ has a \textcolor{blue}{compact} support on $V^c$ which is independent from $\ell$ \blue{when $\ell$ is} large enough.
By convergence in $\cS'$
the sequence $\mathcal F(\varphi u_\ell)$ converges uniformly to
$\mathcal F(\varphi u)$ on it. Then, we deduce the convergence
$$\int_{V^c} \mathcal F(\varphi u_\ell)(\eta) I_{\chi}(\eta) d\eta \to
\int_{V^c} \mathcal F(\varphi u)(\eta) I_{\chi}(\eta) d\eta.$$
Hence, for any $x_0$ in $X$, with $f(x_0)$ in $\pi_Y(\Gamma)$ we have defined an open compact neighborhood $X_0$ of $x_0$ such that for any $\chi$ in $\cS(X_0)$ we define
\begin{equation} \label{formulef*u}
\langle f^{*}u,\chi\rangle  = \int_{V} \mathcal F(\varphi u)(\eta) I_{\chi}(\eta) d\eta +
\int_{V^c} \mathcal F(\varphi u)(\eta) I_{\chi}(\eta) d\eta.
\end{equation}

We extends naturally the construction of $f^*u$ to $\cS(X)$ by linearity and compactness of the support of a Schwartz-Bruhat function. If $u$ is a Schwartz-Bruhat function then by construction $f^*u = u \circ f$.

We prove now for any $u$ in $\cS'(X)$ the inclusion
\begin{equation} \label{eq:inclusion-WF-pullback}
	\WF_{\Lambda}(f^*u)\subset f^* \Gamma.
\end{equation}
Let $(x_0,\xi_0) \notin f^*\Gamma$ and prove that
$(x_0,\xi_0)\notin \WF(f^*u)$.
As before, we write $y_0 = f(x_0)$. If $y_0$ is not in $\pi_Y(\Gamma)$ then $y_0$ is a smooth point of $u$ and by construction $x_0$ is a smooth point of $f^*u$, so $(x_0,\xi_0)$ does not belong to $\WF_{\Lambda}(f^*u)$.
Assume now $y_0$ to be in $\pi_Y(\Gamma)$. We denote by $(f^*\Gamma)_{x_0}$ the image $\: ^{t}Df(x_0)(\Gamma_{y_0})$.
By assumption $\xi_0$ does not belong to $(f^*\Gamma)_{x_0}$.
We consider a $\Lambda$-conical open neighborhood $W$ of $(f^*\Gamma)_{x_0}$, which does not contain $\xi_0$.
As $f$ is strict $\mathcal C^1$, it is also $\mathcal C^1$ and the map $\:^{t}Df:(x,\eta) \mapsto \: ^{t}Df(x)\eta$ is continuous. The inverse image $\:^{t}Df^{-1}(W)$ is open and contains $\{x_0\}\times \Gamma_{y_0}$. Taking $X_0$ and $V$ defined in previous steps,
$(X_0\times V)\cap \:^{t}Df^{-1}(W)$ is open, $\Lambda$-conical and contains $\{x_0\}\times \Gamma_{y_0}$.
We can restrict $X_0$ and $V$ such that $X_0\times V$ is included in $\:^{t}Df^{-1}(W)$.
We consider an open ball $B_{\xi_0}$ centered in $\xi_0$ with radius larger than $\ord \xi_0$ and such that
$B_{\xi_0}\cap W = \emptyset$. In particular, the order of each element of the ball $\xi$ is equal to $\ord \xi_0$.
Thus there is $\varepsilon>0$ such that for any $x$ in $X_0$, for any $\eta$ in $V$, for any $\xi$ in $B_{\xi_0}$ we have
\begin{equation} \label{equalityeps}
|\:^{t}Df(x)(\eta)-\xi|\geq \varepsilon (|\eta| + |\xi|).
\end{equation}
Let $\chi \in \cS(X_0)$. By formulas \ref{formuleIchi} and \ref{formulef*u}, and the change variable formula we have
for any $\xi$ in $B_{\xi_0}$
$$
	\mathcal F(\chi f^*u)(\lambda \xi)  =  \abs{\lambda}^n
	                                 \int_{\eta \in K^n} \mathcal F(\varphi u)(\lambda \eta)
			                  I_{\chi}(\xi, \eta)
					 d\eta
$$
with $$I_{\chi}(\xi, \eta) = \lambda \to \int_{X} \chi(x)\psi_K((x\mid \lambda \xi) -(f(x) \mid \lambda \eta))dx$$
and $\varphi$ is a Schwartz-Bruhat function in $\cS(Y_0)$ equal to 1 on $f(X_0)$.
We show that for any $\xi$ in $B_{\xi_0}$, the function $\lambda \mapsto \mathcal F(\chi f^*u)(\lambda \xi)$ has a bounded support \textcolor{blue}{which is moreover} independent from $\xi$ and we deduce that the point $(x_0,\xi_0)$ is microlocally smooth.
Indeed denoting $r$ the integer $\ord \xi_0 - \min _{x \in supp \chi} \ord\:^{t}Df(x)$ and working separately on a ball $B_r$ with $r$ large enough, on \textcolor{blue}{$V\setminus B_r$} and on $V^c\setminus B_r$ we prove that
$I_\chi(\xi,\eta)$ has a bounded support \textcolor{blue}{which is independent from $\xi$ in $B_{\xi_0}$ and $\eta$ in $K^n$}. For any $\xi$ in $B_{\xi_0}$, and any $\eta$ in $K^n$ with $\ord \eta >r$ we have 	$$\ord(\:^{t}Df(x)\eta - \xi) = \ord \xi = \ord \xi_0$$ then by the sationnary phase formula,
$I_\chi(\xi,\eta)$ has a bounded support \textcolor{blue}{which is} independent from \textcolor{blue}{$\xi$ in $B_{\xi_0}$} and $\eta$ in $B_r$.
As well, by \ref{equalityeps} and the stationary phase formula,
the function $I_\chi(\xi,\eta)$
has a compact support \textcolor{blue}{which is} independent from \textcolor{blue}{$\xi$ in $B_{\xi_0}$ and $\eta$ in $V\setminus B_r$}.
Finally, for any $\eta$ in $V^c\setminus B_r$, by assumption on $\WF_\Lambda(u)$, $(y_0,\eta)$ is a micro-locally
smooth point, in particular $\lambda \to \mathcal F(\varphi u)(\lambda \eta)$ has a bounded support, with a bound locally independent from $\eta$ in $V^{c}\cap S_{\Lambda}^{(n)}$. Again by compactness, we deduce a uniform bound in $\eta$ in $V^{c}\cap S_{\Lambda}^{(n)}$.
Hence, $\lambda \to \mathcal F(\varphi u)(\lambda \eta)$ has a bounded support \textcolor{blue}{which is} independent from
$\eta \in V^c\setminus B_r$.

We end the proof by the continuity of $f^*$. Let $(u_j)$ be a sequence of $\cS'_{\Gamma, \Lambda}(Y)$ which converges to $u$ in
$\cS'_{\Gamma, \Lambda}(Y)$. Let's prove that $(f^* u_j)$ converges to $f^*u$ in $\cS'_{f^*\Gamma, \Lambda}(X)$.
By construction of $f^*u$ and its independance from the sequence $(u_j)$, for any $\chi$ in $\cS(X)$, $(\langle f^*u_j, \chi\rangle )$ converges to $\langle f^*u,\chi\rangle $.
We deduce the second point of the definition of the $\cS'_{f^*\Gamma, \Lambda}$-convergence from the definition of the $\cS'_{\Gamma, \Lambda}(Y)$-convergence of $u_j$ and the proof of $I_\chi(\xi,\eta)$ has a bounded support \textcolor{blue}{which is} independent from \textcolor{blue}{$\xi$ in $B_{\xi_0}$} and $\eta$ in $B_r$ or $V \setminus B_r$.
\end{proof}

\begin{prop}[Functoriality of the pull-back]
Let $f:X\subset K^n\to Y\subset K^m$ and $g : Y \subset K^m \to Z \subset K^\ell$ be strict $C^1$ morphisms of strict $C^1$ submanifolds.  Let $\Gamma$ be a $\Lambda$-cone in $T^*Z\setminus\{0\}$ such that
$$
N_g\cap \Gamma = \emptyset\:\:\text{and}\:\:N_f \cap g^*\Gamma = \emptyset.
$$
Then, the intersection $N_{g\circ f} \cap \Gamma$ is empty and we have the equality
$$(g \circ f)^* = f^* \circ g^*$$
where
$$
\begin{array}{rcl}
(g\circ f)^* & : & \cS'_{\Gamma,\Lambda}(Z) \to \cS'_{(g\circ f)^*\Gamma,\Lambda}(X), \\
g^* & : & \cS'_{\Gamma,\Lambda}(Z) \to \cS'_{g^*\Gamma,\Lambda}(Y),\\
f^* & : & \cS'_{g^*\Gamma,\Lambda}(Y) \to \cS'_{f^*(g^*\Gamma),\Lambda}(X),
\end{array}
$$
are given by the pull-back theorem applied to $(g\circ f)$, $f$ and $g$.
\end{prop}

\begin{proof}
	Assume $N_{g\circ f} \cap \Gamma$ to be non empty and take $(z,\eta)$ an element. By definition, there is $x\in X$ such that
	$z=(g\circ f)(x)$ and $\:^{t}Df_x\left(\:^tDg_{f(x)}(\eta)\right)=0$. In particular, the point
	$\left(f(x),\:^tDg_{f(x)}(\eta)\right)$ belongs to $N_f$. By definition this point belongs also to the cone
	$g^* \Gamma$ which induces a contradiction with the fact $N_f \cap g^* \Gamma = \emptyset$. We conclude that
	$N_{g\circ f} \cap \Gamma = \emptyset$. Hence, by the pull-back theorem, there is a unique continuous map
	$(g\circ f)^*$ from $ \cS'_{\Gamma,\Lambda}(Z)$ to $\cS'_{(g\circ f)^*\Gamma,\Lambda}(X)$ such that for any
	Schwartz-Bruhat function $u$ in $\cS(Z)$, we have $(g\circ f)^*(u)=u \circ (g \circ f)$. Furthermore, by assumptions and the pull-back theorem, the pull-back morphisms $f^*$ and $g^*$ are well defined, the map
	$f^* \circ g^*$ from $ \cS'_{\Gamma,\Lambda}(Z)$ to $\cS'_{(g\circ f)^*\Gamma,\Lambda}(X)$ is continuous and for any Schwartz-Bruhat function $u$ in $\cS(Z)$, we have $f^*\circ g^*(u)=u \circ (g \circ f)$. By uniqueness, we obtain the equality $(g\circ f)^*=f^* \circ g^*$. Finally, using the definition we have the equality
	$$f^{*}\left(g^{*} \Gamma\right) = (g\circ f)^{*}\Gamma.$$
\end{proof}

\begin{cor} \label{cor:invariancebydiffeo}
Let $f:X\subset K^n\to Y\subset K^m$ be a strict $C^1$ isomorphisms between strict $C^1$ submanifolds.
Then the cone $N_f$ is empty and for any distribution $u \in \cS'(Y)$ the pull-back $f^*u$ given by the pull-back theorem
satisfies
$$ \WF_\Lambda(f^*u) = f^* (\WF_\Lambda(u)).$$
\end{cor}
In fact this corollary follows already from Remark \ref{rem:WFcharts}. Alternatively, one can use Theorem \ref{thm:pull} as follows.
\begin{proof}[Proof of Corollary \ref{cor:invariancebydiffeo}]
Since $f$ is a strict $C^1$ isomorphism, the submanifolds $X$ and $Y$ have same dimension and for any point $x$ in $X$ the differential $Df(x)$ is invertible and $N_f$ is necessarly empty. For any distribution $u$ in $\cS'(Y)$, taking
	$\Gamma = \WF_\Lambda(u)$, the pull-back $f^*u$ given by the pull-back theorem satisfies
        $$ \WF_\Lambda(f^*u) \subset f^* (\WF_\Lambda(u)).$$
	As well, we have
	$$ \WF_\Lambda\left((f^{-1})^{*}(f^*u)\right) \subset (f^{-1})^* (\WF_\Lambda(f^*u))$$
	namely by functoriality of the pull-back
	$$ \WF_\Lambda(u) \subset (f^{-1})^* (\WF_\Lambda(f^*u))$$
	and applying $f^*$ we get
	$$ f^* (\WF_\Lambda(u)) \subset \WF_\Lambda(f^*u).$$
\end{proof}

\begin{thm}[Push-forwards of distributions]\label{thm:push:WF}
Let $f:X\subset K^n\to Y\subset K^m$ be a strict $C^1$ morphism between strict $C^1$ submanifolds and let $u$ be a distribution on $X$. Suppose that the restriction of $f$ to the support of $u$ is proper.
Then one has
$$
\WF_\Lambda(f_* u) \subset f_*(\WF_\Lambda(u))\cup N_f.
$$
\end{thm}
\begin{proof}
	We follow \cite[Theorem 2]{Gabor}. By the definition in paragraph \ref{def:pushforward}, the push-forward operation is functorial.
	Hence, for $f:X\subset K^n\to Y\subset K^m$ a strict $C^1$ morphism between strict $C^1$ submanifolds,
	we can decompose $f$ as a composition $p\circ i$ where $i:X\to \Gamma_f$ is the imbedding from $X$ to the
	graph $\Gamma_f$ of $f$, mapping $x$ to $(x,f(x))$ and $p:\Gamma_f \to Y$ is the second projection to $Y$.
	By functoriality, we get $f_* = p_* \circ i_*$. It is then enough to prove the result for imbeddings and projections.\\

	$\bullet$ The imbedding case.
	The wave front set is defined locally and by \ref{cor:invariancebydiffeo} it is invariant by change of coordinates. It is then enough to prove the result in the case of
	$$
		i : \begin{cases} X  \to  X \times Y \\
		   x  \mapsto (x,0).
		  \end{cases}
	$$
	in that case, for any distribution $u$ in $\cS'(X)$, the push-forward $i_*u$ is equal to $u \otimes \delta$,
	where $\delta$ is the Dirac measure,
	$$i_*(\WF_\Lambda(u)) = \{(x,0),(\eta,\xi) \mid (x,\eta) \in \WF_u\} =
	\WF_\Lambda u \times \left(\{0\}\times K^m\right).$$
	and using the theorem \ref{thm:WFtenseur} and the definition of a $\Lambda$-micro-locall-smooth point we obtain the equality
	$\WF_\Lambda(i_*u) = i_*(\WF_\Lambda u).$

	$\bullet$ The projection case. Denote by $\pi$ the canonical projection from $X\times Y$ to $Y$.
	Let $u$ be a distribution in $\cS'(X\times Y)$ and assume the restriction of $\pi$ to the support of $u$ is proper. Let's prove the inclusion
	$$\WF_\Lambda(\pi_* u) \subset \pi_*(\WF_\Lambda u) \cup N\pi.$$
	Remark that in that case $N_\pi=\emptyset$
	and
	$$\pi_*(\WF_\Lambda u) = \{(y,\xi)\in T^*Y \setminus \{0\} \mid \exists (x,\eta) \in T^*X,
        \:((x,y),(\eta,\xi))\in \WF_\Lambda u\}.$$
	Consider $(y_0,\xi_0)$ not in $\pi_*(\WF_\Lambda u)$ with $\xi_0$ in $S_{\Lambda}^{(m)}$.
	Then, for any $(x,\eta)$ in $T^*X$ with $\eta$ in $S_\Lambda^{(n)}$,
	the point $((x,y_0),(\eta, \xi_0))$ does not belong to $\WF_\Lambda(u)$, and there are two neighborhoods
	$U_{(x,y_0),(\eta,\xi_0)}$ of $(x,y_0)$ and $\check U_{(x,y_0),(\eta, \xi_0)}$ of $(\eta,\xi_0)$ such that for all $\phi$
	in $\cS(U_{(x,y_0),(\eta,\xi_0)})$, there is an integer $N_{\phi,(x,y_0),(\eta,\xi_0)}$ such that for all $\lambda$ in $\Lambda$ with
	$\ord \lambda \leq N_{\phi,(x,y_0),(\eta,\xi_0)}$, and $(\eta',\xi')$ in $\check U_{(x,y_0),(\eta,\xi_0)}$,
	\begin{equation} \label{formuleeta'}
	\mathcal F(\phi u)(\lambda(\eta',\xi'))=0.
        \end{equation}
	We consider points $((x,y_0),(\eta, \xi_0))$ with $U_{(x,y_0),(\eta,\xi_0)}$ included in the support of $u$.
	By properness assumption of $\pi$ on the support of $u$, and compactness of the product
	$S_\Lambda^{(n)}\times S_{\Lambda}^{(m)}$, there are finitely many $(x_i,\eta_i)_{i\in I}$ with
	$$\bigcup_{x,\eta} U_{(x,y_0),(\eta,\xi_0)} = \bigcup_{i\in I} U_{(x_i,y_0),(\eta_i,\xi_0)}.$$
	We consider the following open neighborhoods of $y_0$ and $\xi_0$
	$$\Omega_{y_0,\xi_0} = \bigcap_{i \in I} \pi\left(U_{(x_i,y_0),(\eta_i,\xi_0)}  \right) \: \text{and} \:
	\check \Omega_{y_0,\xi_0} = \bigcap_{i \in I} \pi\left(\check U_{(x_i,y_0),(\eta_i,\xi_0)}  \right)
	$$
	We can assume $\Omega_{y_0,\xi_0}$ included in the support of $u$.
	Let $\varphi$ be in $\cS(\Omega_{y_0,\xi_0})$. By definition of $\pi$, the composition $\phi=\varphi \circ \pi$ belongs to the intersection $\bigcap_{i \in I} \cS(U_{(x_i,y_0),(\eta_i,\xi_0)})$.
	For any $\lambda$ in $\Lambda$ with
	$\ord \lambda \leq \max(N_{\phi,(x_i,y_0),(\eta_i,\xi_0)})$, for any $\xi'$ in $\check \Omega_{y_0,\xi_0}$, applying \ref{formuleeta'} for any $\eta_i$, with $i$ in $I$,
	we obtain
	$$ \mathcal F(\varphi \pi_* u)(\lambda \xi') = 0 $$
	and the point $(y_0,\xi_0)$ does not belong to $\WF_\Lambda(\pi_*u)$.
\end{proof}

\begin{example}\label{ex:topology}
In this example we show that the finer topology is needed in order to get continuity of $f^*$ in Theorem \ref{thm:pull}, even already on the subspace $\cS(Y)\subset \cS'_{\Gamma,\Lambda}(Y)$ for $Y=K$.
For each $r\geq 0$, let $v_r$ be the distribution on $Y$ with $Y=K$ given by integration against the Schwartz-Bruhat function
$$
\varphi_r(x) = q_K^{r/2} \11_{B_r(0)}(x),
$$
where $B_r(0)$ is the ball around $0$ of valuative radius $r$. Thus, $v_r(\psi)= \int_K \varphi_r(x) \psi(x)|dx|$ for any Schwartz-Bruhat $\psi$ on $Y$.

Note that the sequence $(v_r)_r$ has a limit which is the zero distribution, since for any Schwartz-Bruhat $\psi$ on $K$ one has
$$
\left| \int_{x\in K} \psi(x) \varphi_r(x) |dx| \right|_\CC \leq  c_\psi q_K^{-r+r/2} \to 0\ \mbox{ for } r\to +\infty,
$$
with $c_\psi= \int_K |\psi(x)|_\CC |dx|$.

Now put $X=K$ and $Y=K$ and let $f:X\to Y$ be the constant function to $0$. Then, the conditions of Theorem \ref{thm:pull} are satisfied to take $f^{*}(v_r)$ for each $r$, with $\Gamma$ any closed cone disjoint from $N_f=\{0\}\times K^\times$.
However, the sequence $w_r:= f^{*}(v_r)$ does not have a limit in $\cS'(K)$, since
$$
w_r (\11_{B_0(0)}) := \int_{x\in K}  \11_{B_0(0)}(x) (\varphi_r\circ f)(x) |dx| =
       \int_{x\in B_0(0)}  q_K^{r/2}  |dx| = q_K^{r/2} \to+\infty.
$$
Of course, the sequence $(v_r)_r$ is not converging in the $\cS'_{\Gamma,\Lambda}$-sense, since it fails condition (\ref{def:S'_Gamma-convergence-Fourier}) of Definition \ref{def:S'_Gamma-convergence}.

The example thus shows that Heifetz construction from \cite{Hei85a} is not continuous with the subset topology on $\cS'_{\Gamma,\Lambda}(X)$ induced from $\cS'(X)$.  Thus, the finer topology on $\cS'_{\Gamma,\Lambda}(X)$, as specified in Definition \ref{def:S'_Gamma-convergence} (similar to the real case by H\"ormander \cite{Hormander83}) is necessary. Heifetz omits to specify this finer topology in \cite{Hei85a}.
\end{example}

We finish this section by the version in our context of the classical result on the product of distributions (see \cite[Theorem 8.2.10]{Hormander83}).
	\begin{thm}
		Let $X$ be a $C^1$ strict submanifold of $K^n$ for some $n\geq 0$. Let $u$ and $v$ be distributions in $\cS'(X)$.
		Then, the product $uv$ can be defined as the pull-back of the tensor product $u\otimes v$ by the diagonal map
		$\delta : X \to X \times X$ unless $(x,\xi) \in \WF_\Lambda(u)$ and $(x,-\xi)\in \WF_{\Lambda}(v)$ for some $(x,\xi)$ in $T^*X\setminus \{0\}$. When the product is defined we have
		$$
\WF^{0}_{\Lambda}(uv)=\{(x,\xi+\eta) \mid (x,\xi) \in \WF^{0}_\Lambda(u),\ (x,\eta) \in \WF^{0}_\Lambda(v)\},$$
with notation from just above Theorem \ref{thm:WFtenseur}.
	\end{thm}
\begin{proof} The result follows from Theorems \ref{thm:WFtenseur} and \ref{thm:pull} with $f=\delta$.
\end{proof}

\section{Distributions of $\cCexp$-class and wave front sets}\label{sec:distCexpclass}

\subsection{}

In this section we introduce a class of distributions given by uniform, field-independent descriptions, called distributions of $\cCexp$-class (see the definition in \ref{Dis}). These distributions are not only uniform, they also have some geometric properties that more arbitrary distributions may lack. We will study the wave front sets associated to distributions of $\cCexp$-class and we will show that such a wave front set is not definable in general (see Example \ref{ex:smooth}), but still it is always the complement of a zero locus of a function of $\cCexp$-class (see Theorem \ref{Smo}), yielding some results on the recent notion of \blue{on $\WF$-holonomicity from \cite{AizDr}} in Section \ref{sec:WF}. Similarly, conditions (on family parameters) related to e.g.~pull-backs are shown to be zero loci of functions of $\cCexp$-class, and the class of $\cCexp$-class distributions is stable under pull-backs (see Theorem \ref{Hei2}). Also, the Fourier transform of distributions of $\cCexp$-class remains of $\cCexp$-class (see Theorem \ref{Fourier:p}).

\subsection{}

From now on, until the end of the paper, we use terminology and notation from Section 1.2 of \cite{CGH5},  without recalling that section in full. In particular this fixes the notions of functions, loci, and conditions of $\cCexp$-class, as well as of definable sets. Thus, $\Loc$ is the collection of all pairs $(F,\varpi)$ of non-archimedean local fields $F$ with
a uniformizer $\varpi$ of the valuation ring $\cO_F$ of $F$. 
Further, given an integer $M$, $\Loc_M$ is the collection of $(F,\varpi) \in \Loc$ such that
$F$ has characteristic either $0$ or at least $M$.
Given $F=(\underline F,\varpi) \in \Loc$, we write $\VF_F$ (or, by abuse of notation, just $F$) for the valued field $\underline F$, we write $\cO_F$ for the valuation ring of $\VF_F$, $\cM_F$ for the maximal ideal, $\RF_F$ for the residue field and $q_F$ for
the number of elements of $\RF_F$. The value group (even though always being isomorphic to $\ZZ$) will sometimes be denoted by $\VG_F$, and,
for positive integers $n$, we write $\RF_{n,F}$  for the quotient $\cO_F/n\cM_F$ and $\ord_F\colon \VF_F \to \VG_F \cup \{\infty\}$ for the valuation map. Apart from the natural operations like the ring operations and the valuation map, note that definable sets are built up with a generalized Denef-Pas language $\gLPas$ involving also angular
component maps $\ac_n:\VF\to \RF_n$ for each integer $n>0$.

We write $\cD_F$ for the subset of the group of additive characters $\psi$ on $F$, such that $\psi(\cM_F)=1$ and $\psi(\cO_F)\not=1$, with $\cM_F$ the maximal ideal of $\cO_F$.
We introduce one handy extra abbreviation on top of the notation of \cite[Section 1.2]{CGH5}: By $\Loc'$ we mean the collection of $F=(\underline F,\varpi,\psi)$ with $(\underline F,\varpi)$ in $\Loc$ and $\psi\in \cD_{F}$. We use the obvious variants as for $\Loc$ in \cite{CGH5}, like $\Loc'_M$ (to denote that the characteristic of the local field is $0$ or at least $M$), and $\Loc'_{\gg 1}$ (to denote that the characteristic of the local field is $0$ or at least $M$ for some $M$ depending on the context).

For arbitrary sets $A\subset X\times T$ and $x\in X$, write $A_x$ for the set of $t\in T$ with $(x,t)\in A$.
For $g:A\subset X\times T\to B$ a function
and for $x\in X$, we write $g(x,\cdot)$ or $g_x$ for the function $A_x\to B$ sending $t$  to $g(x,t)$.

\subsection{Distributions of $\cCexp$-class}

\begin{def-prop}[Distributions of $\cCexp$-class]\label{Dis} Consider definable sets $Y$ and $W\subset Y\times \VF^n$ for some $n\geq 0$. Let $E$ be in $\cCexp(W\times \ZZ)$.
For each $F\in \Loc'_{\gg 1}$, consider the set
$$
{\rm {Dis}}(E,Y)_F:= \{y\in Y_F\mid W_{F,y} \mbox{ is a strict $C^1$ submanifold of $F^n$, and, }
$$
$$
E_F(y,\cdot) : W_{F,y}\times \ZZ\to \CC \mbox{ is a $B$-function}\}.
$$
 Then ${\rm {Dis}}(E,Y)$ is a $\cCexp$-locus. In particular, there is $M>1$ such that for $F\in\Loc'_M$ and $y\in {\rm {Dis}}(E,Y)_F$, the function $E_F(y,\cdot)$ is the $B$-function of a distribution $u_{F,y}$ on $W_{F,y}$. The collection of distributions $u_{F,y}$, for $F\in\Loc'_M$ and $y\in {\rm {Dis}}(E,Y)_F$, is called a $\cCexp$-class distribution.
\end{def-prop}
\begin{proof}
The condition that $W_{F,y}$ is a strict $C^1$ submanifold of $F^n$ is clearly a definable condition, and hence, certainly a $\cCexp$-condition.
Since a finite intersection of $\cCexp$-loci is a $\cCexp$-locus, we can proceed as follows.
First we show that the collection, for $F\in \Loc'_{\gg 1}$, of the sets
$$
{\rm {Ba}}(E,Y)_F:= \{y\in Y\mid W_{F,y} \mbox{ is a strict $C^1$ manifold, and,}
$$
$$
E_F(y,\cdot):W_{F,y}\times \ZZ \to \CC \mbox{ is a function on  balls}\}
$$
is a $\cCexp$-locus.

That $E_F(y,\cdot)$ is a function on balls can be expressed as
\[\forall x, x' \in W_{F,y}:\forall r \in \ZZ:\]
\[ \left((B_r(x)\cap W_y=B_r(x')\cap W_y)\quad \Rightarrow\quad E_F(y,x,r) = E_F(y,x',r) \right);\]
by the formalism explained after \cite[Proposition 1.3.1]{CGH5}, this is a $\cCexp$-locus, and hence so is ${\rm {Ba}}(E,Y)$.
In more detail:
Consider the function $E_2$
$$
E_2:Y\times \VF^{2n}\times \ZZ \to \ZZ
$$
sending $(y,x,x',r)$ to $E(y,x,r) - E(y,x',r)$ whenever $x,x'$ lie in $W_y$ and moreover $B_r(x)\cap W_y=B_r(x')\cap W_y$
and $W_{F,y}$ is a strict $C^1$ manifold, and to $0$ otherwise. Then clearly $E_2$ lies in $\cCexp(Y\times \VF^{2n}\times \ZZ)$ and
$$
{\rm {Ba}}(E,Y) = \{y\in Y | W_{F,y}  \mbox{ is a strict $C^1$ manifold, and,}
$$
$$
\forall x\forall x' \forall r\ E_2(y,x,x',r) =0\},
$$
which is a $\cCexp$-locus by Proposition 1.3.1(3) of \cite{CGH5}. 

Finally we show that ${\rm {Dis}}(E,Y)$ is a $\cCexp$-locus. For a function on balls, in order for (\ref{add}) to hold in general it is enough to check (\ref{add}) for an arbitrary ball $B$ in $\VF^n$, say, of valuative radius $r$, and for $B_i$ disjoint balls in $B$ of valuative radius $r+1$ whose union equals $B$. By Theorem 4.1.1 of \cite{CHallp}, there are functions $H$ and $G$ in $\cCexp(Y\times \VF^n\times \ZZ)$ such that
\begin{equation}\label{intE}
H(y,x,r) = \int_{x'\in B_r(x)\cap W_y}  q^{\blue{(r+1)\dim W_y}} E(y,x',r+1) \blue{\mu_{W_y}}
\end{equation}
whenever \blue{$W_y$ is a  strict $C^1$ manifold and the integral is finite, and such that
$G(y,x,r)=0$ if and only if these conditions hold (namely, the integral in (\ref{intE}) is finite and $W_y$ is a  strict $C^1$ manifold). Here, $\mu_{W_y}$ stands for the canonical measure on $W_y$ as in Section \ref{def:pushforward},} and, $q$ stands for the $\cCexp$-function which associates to $F$ the number of residue field elements $q_F$, with notation as in \cite{CGH5}.
Now let $E_3$ be the product of $H - E$ with the characteristic function of tuples $(y,x,r)$ such that $W_y\cap B_r(x)$ is compact; note that $W_y\cap B_r(x)$ being compact is a definable condition.  (Here, we have extended $E$ by zero outside its original domain.)
By construction and lemma \ref{lem:add}, $y\in Y$ lies in ${\rm {Dis}}(E,Y)$ if and only if, jointly, $y$ lies in ${\rm {Ba}}(E,Y)$ and, for all $x\in \VF^n$ and all $r\in \ZZ$, $G(y,u,r)=0$ and $E_3(y,u,r)=0$. It now follows from Proposition 1.3.1 of \cite{CGH5} that ${\rm {Dis}}(E,Y)$ is a $\cCexp$-locus.
\end{proof}

Note that a distribution of $\cCexp$-class is in fact, by its definition in \ref{Dis}, a family of distributions \blue{whose (family of) $p$-adic continuous wavelet transforms is of $\cCexp$-class.}


\begin{prop}[Families of Schwartz-Bruhat functions of $\cCexp$-class]\label{alpha+-:p}
Consider definable sets $Y$ and $W\subset Y\times \VF^n$ for some $n\geq 0$.
Let $\varphi$ be in $\cCexp(W)$. Then, 
the collection for $F\in \Loc'_{\gg 1}$ of the sets
$$
{\rm {Sch}}(\varphi,Y)_F := \{y\in Y_F\mid W_{F,y} \mbox{ is a strict $C^1$ manifold, and, }
$$
$$
\varphi_F(y,\cdot) \mbox{ is a Schwartz-Bruhat function on } W_{F,y}\}
$$
is a $\cCexp$-locus.
Furthermore, there are definable functions $\alpha^+$ and $\alpha^-$ from $Y$ to $\ZZ$ and $M>0$ with the following properties for any $F\in\Loc'_M$ and any $y\in  {\rm {Sch}}(\varphi,Y)_F$.
The function $\varphi_F(y,\cdot)$ is constant on any set of the form $W_{y}\cap B$ with $B\subset F^n$ a ball of valuative radius $\alpha^+_F(y)$.
Furthermore, $\varphi_F(y,\cdot)$ is supported on the ball around zero of valuative radius $\alpha^-_F(y)$.
\end{prop}
\begin{proof}
As before, we note that the condition that $W_{F,y}$ is a strict $C^1$ submanifold of $\VF_F^n$ is a definable condition, and hence, a $\cCexp$-condition. By Corollary 1.4.3 of \cite{CGH5}, the condition that $\varphi_F(y,\cdot)$ is locally constant on $W_{F,y}$ is a $\cCexp$-locus, and Proposition 1.4.1 of \cite{CGH5} yields the existence of a definable map $\alpha^+$ with properties as desired, namely, such that if $\varphi_F(y, \cdot)$ is Schwartz-Bruhat, then $\varphi_F(y, \cdot)$ is constant on any set of the form $W_{y}\cap B$ with $B\subset F^n$ a ball of valuative radius $\alpha^+_F(y)$.

Finally, $\varphi_F(y,\cdot)$ is Schwartz-Bruhat if and only if the following conditions hold: The function $\varphi_F(y,\cdot)$ is constant on $W_{F,y}$ intersected with any ball of valuative radius $\alpha^+_F(y)$ and moreover
\[
\forall \lambda \ll 0: \forall x \in W_{F,y} \setminus B_\lambda(0): \varphi_F(y,x) = 0
\]
\[\wedge
\]
\[
\forall x \in W_{F,y}:\left( (B_{\alpha^+_F(y)}(x) \cap W_{F,y}\text{ is not compact}) \Rightarrow \varphi_F(y,x) = 0\right)
\]
(here, $\lambda$ runs over all sufficiently small integers). This is a $\cCexp$-locus by the formalism
explained after  \cite[Proposition 1.3.1]{CGH5} and using \cite[Proposition 1.4.1]{CGH5} to treat ``$\forall \lambda \ll 0$''.
Also, \cite[Proposition 1.4.1]{CGH5} yields a ``definable witness for $\lambda$'', i.e., the desired function $\alpha^-_F(y)$.
\end{proof}

\begin{lem}[Convolutions]\label{lem:conv}
Consider a definable set $Y$ and some $n\geq 0$.
Let $\varphi$ and $\psi$ be in $\cCexp(Y\times \VF^n)$.
Then there are $M>0$ and $E$ in $\cCexp(Y\times \VF^n)$ such that, for any $F\in\Loc'_M$, any $y\in Y_F$ and any $x\in \VF_F^n$,
$$
E_{F}(y,x) = \int_{z\in F^n} \varphi(y,x-z)\psi(y,z)|dz|
$$
whenever this integral is finite.
\end{lem}
\begin{proof}
This follows directly from the stability under integration given by Theorem 4.1.1 of \cite{CHallp}.
\end{proof}


The evaluation of distributions of $\cCexp$-class in families of Schwartz-Bruhat functions of $\cCexp$-class is again of $\cCexp$-class, as follows.

\begin{prop}\label{Julia'sdef:p}
Consider definable sets $Y$ and $W\subset Y\times \VF^n$ for some $n\geq 0$. Let $E$ be in $\cCexp(W\times \ZZ)$ and let $\varphi$ be in \blue{$\cCexp(W)$}.
Then there are $M>0$ and a function $J$ in $\cCexp(Y)$ such that for each $F$ in $\Loc'_M$ and each $y\in Y_F$ one has
$$
J_F(y) = u_{F,y} (\varphi_{F,y})
$$
whenever $W_{F,y}$ is a strict $C^1$ manifold, $E_{F,y}$ is a $B$-function on $W_{F,y}$ of a distribution denoted by $u_{F,y}$, and $\varphi_{F}(y,\cdot)$ is a Schwartz-Bruhat function on $W_{F,y}$, denoted by $\varphi_{F,y}$.
\end{prop}
\begin{proof}
Given $\varphi$, take definable functions $\alpha^+$ and $\alpha^-$ on $Y$ as in Proposition \ref{alpha+-:p}.
By the stability under integration of $\cCexp$-class function from Theorem 4.1.1 of \cite{CHallp}, there is a function \blue{$J$} in $\cCexp(Y)$ such that
$$
J(y) =   q^{\alpha^+(y) \blue{\dim W_y}} \int_{x\in B_{\alpha^-(y)} (0) \blue{\cap W_y} }   \varphi\blue{(y,x)} E(y,x,\alpha^+(y)) \blue{\mu_{W_y}},
$$
whenever \blue{$W_y$ is a  strict $C^1$ manifold and the integral is finite}, and where \blue{$\mu_{W_y}$ stands for the canonical measure on $W_y$ as in Section \ref{def:pushforward} and $q$ stands for the $\cCexp$-function which associates to $F$ the number of residue field elements $q_F$.} Then $J$ is as desired.
\end{proof}

By the previous proposition, we can evaluate $\cCexp$-class distributions on $\cCexp$-class Schwartz-Bruhat functions with outcome being again of the same class. The next result gives stability under Fourier transform.


\begin{thm}\label{Fourier:p}
The Fourier transform of a $\cCexp$-class distribution on $\VF^n$ is of $\cCexp$-class.
More precisely, consider a definable set $Y$ and let $E$ be in $\cCexp(Y\times \VF^n\times \ZZ)$ for some $n\geq 0$. Then there are $M>0$ and $\cF_{/Y}(E)$ in $\cCexp(Y\times \VF^n\times \ZZ)$ such that for each $F\in\Loc'_M$ and each $y\in {\rm {Dis}}(E,Y)_F$ one has that $\cF(E)_{/Y,F}(y,\cdot)$ is the $B$-function of the Fourier transform of the distribution associated to the $B$-function $E_F(y,\cdot)$, with respect to the character $\psi$ if one writes $F=(\underline F,\varpi,\psi)$ (see Definition \ref{def:four}).
\end{thm}
\begin{proof}
This follows from Proposition \ref{Julia'sdef:p} applied to $E(y,\cdot)$ and the collection of the Schwartz-Bruhat functions which are the Fourier transform of characteristic functions of balls, which is of $\cCexp$-class by Theorem 4.1.1 of \cite{CHallp}.
\end{proof}

\subsection{The nature of wave front sets of $\cCexp$-class distributions}


By Example \ref{ex:smooth} below, and in the spirit of Theorem \ref{WFu=S}, the smooth locus of a $\cCexp$-class distribution is not always a definable set (it can in fact be far more general). The following theorem is one of our main results, namely that the micro-locally smooth locus of a $\cCexp$-class distribution is a $\cCexp$-locus, and thus, that the wave front set is the complement of the zero locus of a $\cCexp$-class function.

\begin{thm}[Loci of microlocal smoothness and zero loci]\label{Smo} Let $\Lambda\subset\VF$, $Y$, and $W\subset Y\times \VF^n$ be definable sets for some $n\geq 0$. Let $E$ be in $\cCexp(W\times \ZZ)$.
Define the sets, for $F$ in $\Loc'_{\gg 1}$,
$$
\Smo(\Lambda,E,Y)_F := \{(y,x,\xi)\in W_F\times F^n\mid y\in {\rm Dis}(E,Y)_F \mbox{  and }
$$
$$
\mbox{$\Lambda_F$ is an open subgroup of finite index in $F^\times$ and }
$$
$$
\mbox{the distribution associated to $E_{F,y}$ is $\Lambda_F$-smooth at $(x,\xi)$}     \}.
$$
Then $\Smo(\Lambda,E,Y)$ is a $\cCexp$-locus. Here, we have considered $T^*W_{F,y}$ inside $F^n\times F^n$ when $W_{F,y}$ is a strict $C^1$ submanifold of $F^n$.
In particular, if $Y$ is a point and $E_F$ is a distribution on a strict $C^1$ manifold $W_F$ for each $F$, then
$(\WF_{\Lambda} (E_F))_F$ is the complement of a $\cCexp$-locus.
\end{thm}

Note that the complement of a $\cCexp$-locus is not always a $\cCexp$-locus itself.
Before proving Theorem \ref{Smo}, we slightly rephrase Definition \ref{Hei:smooth}.
\begin{lem}\label{smooth}
Let $F$ be a non-archimedean local field and let $\Lambda$ be an open subgroup of finite index in $F^\times$.
Let $u$ be a distribution on $F^n$. Let $(x_0,\xi_0)$ be in $T^*F^n\setminus\{0\}$. Then, $u$ is $\Lambda$-smooth at $(x_0,\xi_0)$ if and only if, for sufficiently large $r\in \ZZ$, for all $x\in B_r(x_0)$ and all $s>r$, one has the following. For sufficiently large $\lambda\in\Lambda$ (namely with large norm) and all $\xi\in B_r(\xi_0)$ one has
$$
\cF(\11_{B_s(x)} u) (\lambda\xi) = 0.
$$
\end{lem}
\begin{proof}
Since Schwartz-Bruhat functions are finite linear combinations of characteristic functions of balls and since any neighborhood of $x_0$ contains an open ball of the form $B_r(x_0)$, the lemma follows.
\end{proof}

\begin{proof}[Proof of Theorem \ref{Smo}]
Again we use that intersections of $\cCexp$-loci are $\cCexp$-loci.
For a definable subset $\Lambda_F$ of $F^\times$ to be a subgroup is clearly a definable condition. Given that $\Lambda_F$ is a subgroup of $F^\times$, it is open if and only if it contains an open neighborhood of $1$, and, it is of finite index if and only if for any integer $N>0$ there is $\lambda\in\Lambda_F$ with $\ord (\lambda)>N$. Hence, the joint conditions on $\Lambda_F$ that it be an open subgroup of finite index in $F^\times$ is a definable condition, uniformly for $F$ in $\Loc'$.

We now first treat the special case that $W_{F,y}=F^n$ for all $F$. We proceed by unwinding Lemma \ref{smooth}. For $y\in {\rm {Dis}}(E,Y)$, and $u_y$ the distribution associated to $E(y,\cdot)$,
$\cF(\11_{B_s(x)}  u_y) (\lambda\xi) = 0$ is clearly a $\cCexp$-locus condition on the involved variables $s,x,y,\lambda,\xi$. Also, by Proposition 1.3.1 and 1.4.1 of \cite{CGH5}, the condition `for all sufficiently large $\lambda\in\Lambda$' is a $\cCexp$-locus condition, as are the conditions `$s>r$', `for all $x\in B_r(x_0)$', and, `for sufficiently large $r$'. Putting this together we find that the condition of $\Lambda$-smoothness on $(x_0,\xi_0)$ is a $\cCexp$-locus condition on $(x_0,\xi_0)$ when $W_{F,y}=\VF_F^n$. The case of general $W_{F,y}$ follows from this by taking charts, which can be taken to be coordinate projections by the following property. Let $F$ be a non-archimedean local field, let $X\subset F^n$ be a strict $C^1$ submanifold of dimension $m$ and let $x$ be in $X$. Then there exists an open neighborhood $O$ of $x$ in $X$ and a subset $I$ of $\{1,\ldots,n\}$ with $m$ elements, such that the coordinate projection $p_I$ sending $x=(x_1,\ldots,x_n)$ to $(x_{i})_{i\in I}$, gives a strict $C^1$ chart for $O$ (the existence of such $O$ follows from the Jacobian property Theorem 5.3.1 of \cite{CHallp}).
\end{proof}

%



The previous theorem can be understood as a description of wave front sets as complements of $\cCexp$-loci.

Our next main result gives that pull-backs of $\cCexp$-class distributions are still of $\cCexp$-class. Recall the notation $N_f$ of Section \ref{sec:sequences}
for a strict $C^1$ map $f:X\to Z$ of strict $C^1$ manifolds, with
$$
N_f = \{ (z , \eta) \in T^*Z\setminus\{0\} \mid \exists x\in X \mbox{ with $f(x)=z$ and } 0={}^tDf(x)\eta \}.
$$

In Theorem \ref{Hei2}, $f^*(E)$ will stand for the pull-back of the distribution $E$ of $\cCexp$-class.


%
%
%

\begin{thm}[Pull-backs of $\cCexp$-class distributions]\label{Hei2}
Let $\Lambda\subset\VF$, $Y$, and $W_i\subset Y\times \VF^{n_i}$ be definable sets for $i=1,2$ and some $n_i\geq 0$. Let $E$ be in $\cCexp(W_2\times \ZZ)$.
Consider a definable function $f:W_1\to W_2$ over $Y$.
Then, the collection of sets for $F\in\Loc'_{\gg 1}$
$$
{\rm Pull}(\Lambda,f,E,Y)_F := \{ y\in Y_F \mid y\in {\rm Dis}(E,Y)_F \mbox{  and }
$$
$$
\mbox{ $f_{F,y}$ is a strict $C^1$ morphism between strict $C^1$ manifolds $W_{i,F,y}$, and }
$$
$$
N_{f_{F,y}} \subset \Smo(\Lambda,E,Y)_{F,y} \}
$$
is a $\cCexp$-locus. 
Moreover, there is a function $f^*(E)$ in  $\cCexp(W_1\times \ZZ)$ such that $(f^*(E))_{F,y}$ is the $B$-function of $f_{F,y}^*(u_{E_{F,y}})$ whenever $y$ lies in $Pull(\Lambda,f,E,Y)_F$, and where $u_{E_{F,y}}$ is the distribution associated to $E_F(y,\cdot)$.
\end{thm}
\begin{proof} Clearly the property that $f_{F,y}$ is a strict $C^1$ morphism between strict $C^1$ manifolds $W_{i,F,y}$ is a definable condition on $y\in Y_F$ and $F\in \Loc'_{\gg 1}$. Hence, we may focus on the condition
\begin{equation}\label{eq:inc}
N_{f_{F,y}} \subset \Smo(\Lambda,E,Y)_{F,y},
\end{equation}
for the relevant $y$ and $F$, and on the existence of $f^*(E)$ in $\cCexp(W_1\times \ZZ)$.
Write $u_{F,y}$ for $u_{E_{F,y}}$.
Clearly $N_{f_{F,y}}$ is a family of definable sets, say a definable set $N_f$ with $N_{f,F,y}=N_{f_{F,y}}$.
Consider $\Smo(\Lambda,E,Y)_F$, and let it be the locus of $g$ in $\cCexp(W_{2,F}\times F^n)$ as given by Theorem \ref{Smo}.
The inclusion condition (\ref{eq:inc}) holds if and only if the product of the characteristic function of $N_{f}$ with $g$ is identical zero. Since the $\forall$ quantifier can be eliminated from $\cCexp$-locus conditions by Proposition 1.3.1 of \cite{CGH5}, it follows that ${\rm Pull}(\Lambda,f,E,Y)$ is a $\cCexp$-locus. Let us now construct $f^*(E)$ as desired.
Take $G$ in $\cCexp(W_2\times \ZZ)$, as provided by Lemma \ref{lem:seq} below. So, for each $F$ in $\Loc'_{\gg 1}$ and each $y\in {\rm Pull}(\Lambda,f,E,Y)_F$, the sequence of functions $G_{F,y,i}$ for $i\geq 0$ is a sequences of Schwartz-Bruhat functions such that the associated sequence of distributions $u_{y,i}$ converges to $u_{y} = u_{E_{F,y}}$ for the topology of $\cS'_{\WF_{\Lambda_F}(u_{F,y}),\Lambda_F}(W_{2,F,y})$ as in Definition \ref{def:S'_Gamma-convergence}. For each such $F$ and $y$ and each $i\geq 0$, let $u^*_{y,i}$ be the distribution associated to $G_{F,y,i}\circ f_{F,y}$. Clearly the collection of $B$-functions of the $u^*_{y,i}$ is of $\cCexp$-class, for $F$ in $\Loc'_{\gg 1}$ and $y\in {\rm Pull}(\Lambda,f,E,Y)_F$.
By Theorem 3.1.3 (2) of \cite{CGH5} on limits, there exists $f^*(E)$ in $\cCexp(W_1\times \ZZ)$ such that, for each $F\in \Loc'_{\gg 1}$, each $y\in Y_F$, each $x\in W_{1,F,y}$, and each $r\in\ZZ$,
$f^*(E)_{F}(y,x,r)$ equals $0$ whenever $B_r(x) \cap W_{1,F,y}$ is not compact, and, equals the limit for $i\to +\infty$ of
$$
u^*_{y,i}(\11_{B_r(x)} \cap W_{1,F,y}),
$$
whenever this limit exists and $B_r(x) \cap W_{1,F,y}$ is compact. Then $f^*(E)$ is as desired, by the continuity property of the pull-back as given by Theorem \ref{thm:pull}.
\end{proof}


\begin{lem}[Converging sequences of $\cCexp$-class distributions]\label{lem:seq}
Consider definable sets $\Lambda\subset\VF$,  $Y$ and  $W\subset Y\times \VF^{n}$ for some $n\geq 0$. Let $E$ be in $\cCexp(W\times \ZZ)$.
Suppose for each $y\in {\rm {Dis}}(E,Y)$ that $W_{F,y}$ is a strict $C^1$ submanifold of $F^n$, that $\Lambda_F$ is an open subgroup of finite index in $F^\times$, and that
$\Gamma_{F,y}$ is a $\Lambda_F$-cone. (Note that we require no conditions on the family of cones $\Gamma_{F,y}$.)
Then there exists $\varphi$ in $\cCexp(W\times \ZZ)$ such that
each $\varphi_{F,y,j}$ is a Schwartz-Bruhat function, and such that, for each $y \in {\rm {Dis}}(E,Y)$,
the sequence $\varphi_{F,y,j}$ for $j\in\NN$ converges to $u_{E_{F,y}}$ in the $\cS'$-sense, and if moreover $u_{E_{F,y}}$ belongs to
$\cS'_{\Gamma_{F,y},\Lambda_F}(W_{F,y})$, then it also converges to $u_{E_{F,y}}$ in the $\cS'_{\Gamma_{F,y},\Lambda_F}$-sense,
where we identify the Schwartz-Bruhat functions $\varphi_{F,y,j}$ with the associated distribution and where convergence is as in Definition \ref{def:S'_Gamma-convergence}. 
\end{lem}
\begin{proof}
The existence of $\varphi$ in $\cCexp(W\times \ZZ)$ as desired follows directly from the construction in the proof of Proposition \ref{SB-approx-S':bis}.
\end{proof}

As final result of this section we show that also push-forwards of $\cCexp$-class distributions are of $\cCexp$-class; notationally in Theorem \ref{Push2}, $f_*(E)$ stands for the push-foward of the distribution $E$ of $\cCexp$-class.

\begin{thm}[Push-forwards of $\cCexp$-class distributions] \label{Push2}
Consider definable sets $Y$ and $W_i\subset Y\times \VF^{n_i}$ and for some $n_i\geq 0$ and $i=1,2$ and a definable function $f:W_1\to W_2$ over $Y$.
Let $E$ be in $\cCexp(W_1\times \ZZ)$.
Define, for $F\in \Loc'_{\gg 1}$,
$$
{\rm{Prop}}(f,E,Y)_F := \{ y\in Y_F \mid y\in {\rm Dis}(E,Y)_F
$$
$$
\mbox{ $f_{F,y}$ is a strict $C^1$ morphism between strict $C^1$ manifolds $W_{i,F,y}$, and }
$$
$$
\mbox{ $f_{F,y}$ is proper on the support of $u_{E_{F,y}}$} \}.
$$
Then, the collection ${\rm{Prop}}(f,E,T)_F$ for $F\in \Loc'_{\gg 1}$ is a $\cCexp$-locus, and, there are $M>0$ and $f_*(E)$ in $\cCexp(W_2\times \ZZ)$ such that for all $F\in \Loc'_{\gg 1}$, $(f_*(E))_{F,y}$ is the $B$-function of $f_{F,y,*}(u_{E_{F,y}})$ whenever $y$ lies in ${\rm Prop}(f,E,Y)_F$.
\end{thm}
\begin{proof} Write $u_{F,y}$ for $u_{E_{F,y}}$.
The fact that $f_*(E)$ exists and is of $\cCexp$-class is easy to see. Indeed, for $F$ in $\Loc'_{\gg 1}$, $y$ in ${\rm Prop}(f,E,Y)_F$, $x\in W_{2,F,y}$, and $r\in\ZZ$, put
$$
f_*(E)_{F,y}(x,r) = u_{F,y}(\11_{B_r(x) \cap W_{2,F,y} } \circ f_{F,y})
$$
whenever  $\11_{B_r(x) \cap W_{2,F,y} } \circ f_{F,y}$ is a Schwartz-Bruhat function on $W_{2,F,y}$.
Then $f_*(E)$ is of  $\cCexp$-class by the second part of Proposition \ref{Julia'sdef:p}.

We now show that the collection of sets ${\rm{Prop}}(f,E,Y)_F$ for $F\in \Loc'_{\gg 1}$ is a $\cCexp$-locus.
By the continuity of $f_{F,y}$, for any ball $B$ in $F^{n_2}$ such that $W_{2,F,y}\cap B$ is compact one has that $f_{F,y}^{-1}(W_{2,F,y}\cap B)$  is closed in $W_{1,F,y}$ for the subset topology on $W_{1,F,y}\subset F^{n_1}$. Hence, also the intersection of $f_{F,y}^{-1}(W_{2,F,y}\cap B)$ with the support of $u_{F,y}$ is closed in $W_{1,F,y}$; let us denote this intersection by $A_{y,B}$. Let $\mu_{F,y}:W_{1,F,y}\to \NN$ be a definable function such that, for all $a\in \NN$, $\mu_{F,y}^{-1}(a)$ is compact, and such that $\bigcup_{a\in\NN} \mu_{F,y}^{-1}(a)$ equals $W_{1,F,y}$. We use the following properties:

1) For $x\in W_{1,F,y}$, to lie in the complement of the support of $u_{F,y}$ is a $\cCexp$-locus condition on $x$ and $y$, and hence, to be in the complement of $A_{y,B}$ is also a $\cCexp$-locus condition.

2) Whether for all balls $B$ such that $W_{2,F,y}\cap B$ is compact,
the set $\mu_{F,y}(A_{y,B}) \subset \NN$ is bounded by a constant $C(B) \in \NN$
is a $\cCexp$-locus condition.

3) The map $f_{F,y}$ is proper on the support of $u_{y}$ if and only if for all $B$ such that $W_{2,F,y}\cap B$ is compact,
the set $\mu_{F,y}(A_{y,B}) \subset \NN$ is bounded by a constant $C(B) \in \NN$.

To see why 1) holds, one notes that $x\in W_{1,F,y}$ lies in the complement of the support of $u_{F,y}$ if and only if $E_{F,y}(x,r)=0$ for all large enough integers $r$. Now 1) follows from Proposition 1.4.1 of \cite{CGH5}.  To show 2), one uses Proposition 1.4.1 of \cite{CGH5} once more, in combination with 1). Property 3) is purely topological and is left to the reader.

Now the claim follows using the formalism from \cite{CGH5}, namely as follows.
Since one can also eliminate universal quantifiers by Proposition 1.3.1(3) of \cite{CGH5}, and since quantifying over all balls $B$ such that $W_{2,F,y}\cap B$ is compact can be done definably with universal quantifiers (by using a tuple of variables running over $B$ and a variable for its radius), it follows from 1), 2) and 3) that the collection ${\rm{Prop}}(f,E,Y)_F$ is a $\cCexp$-locus.
 \end{proof}


\begin{example}\label{ex:smooth}
We give an example which shows that the wave front set (of a distibution of $\cCexp$-class) is in general far from a definable set. The situation is even worse, as we show that even the smooth locus is not a definable set in general.

Let $G$ be any function in $\cCexp(\VF^n)$ 
Suppose that $G_F$ is locally integrable over $F^n$ for each $F\in \Loc'_{\gg 1}$. Then, for each $F\in \Loc'_{\gg 1}$, we can consider $G_F$ as a distribution on $F^n$.
For each $F\in \Loc'_{\gg 1}$, let $Y_F$ be the locus where $G_F$ is locally constant. Note that in general the collection $Y_F$ for $F\in \Loc'_{\gg 1}$ is not a definable set at all. For example, one can take $n=3$ and $G_F(x,y,z) = (q_F^{\ord x}  - \ord y) \ord z$ for nonzero $(x,y,z)\in F^3$, extended by $0$.
Nevertheless, the smooth locus of $G_F$ is precisely $Y_F$.

It would be interesting to characterise exactly the sets that can appear as the sets of smooth points, and as the sets of $\Lambda$-micro-locally smooth points. So far, we know by Theorem \ref{Smo} that they are zero loci of functions of $\cCexp$-class (with conic structure).

\blue{Also, it would be interesting, for  wave front sets of special distributions, to check whether they are definable sets, like for push-forwards under polynomial mappings of smooth distributions, their Fourier transforms, etc. Note that the wave front set of a push-forward of a smooth distribution under a polynomial mapping $f$ is contained in $N_f$ by Theorem \ref{thm:push:WF} and that $N_f$ itself is a definable set. }
\end{example}

\section{$\WF$-holonomicity}\label{sec:WF} 

\subsection{}\label{sec:WF1}
In this section we elaborate on a notion of \blue{algebraic} $\WF$-holonomicity introduced by Aizenbud and Drinfeld in \cite{AizDr}.
We rephrase a question by Aizenbud and Drinfeld for future research, based on Remark 3.2.2 and the surrounding text of \cite{AizDr}, and we make a first step in its direction in Theorem \ref{thm:AizDrthmA}.

We first adapt the definition of \cite[Section 3.2, 3.2.1 and 3.2.2]{AizDr} of \blue{algebraic} $\WF$-holonomicity, to the situation of any non-archimedean local field $F$ (of arbirary characteristic).

Say that a Zariski closed subset $C$ of $\AA^n$ has dimension $-\infty$ if it is empty, and has dimension $i$ for some $i\geq 0$ when one of the irreducible components of $C$ has dimension $i$ and any irreducible component of $C$ has dimension $\leq i$.

\begin{defn}\label{def:WF-hol}
Let $F$ be in $\Loc$ and let $X\subset F^m$ be a strict $C^1$ variety of dimension $n$. Let $u$ be a distribution on $X$, and let $\WF(u)$ be its $F^\times$-wave front set.

Say that $u$ is \blue{algebraically} $\WF$-holonomic if \blue{$X=\cX(F)$ for some smooth subvariety $\cX$ of $\AA^m_F$ and there are finitely many smooth locally closed subvarieties $\cY_i$ of $\cX$ such that $\WF(u)$ is contained in
$$
\bigcup_i CN_{\cY_i}^{\cX}(F),
$$
where $CN_{\cY_i}^{\cX}$ is the co-normal bundle of  $\cY_i\subset \cX$.}

\blue{Say that $u$ is $F$-analytically $\WF$-holonomic if $X$ is an $F$-analytic submanifold of $F^m$ and there are finitely many $F$-analytic submanifolds $Y_i$ of $X$ such that $\WF(u)$ is contained in the union over $i$ of the co-normal bundles $CN_{Y_i}^X$ of  $Y_i\subset X$.

Say that  $u$ is strictly $C^1$ $\WF$-holonomic if there are finitely many strict $C^1$ subvarieties $Y_i$ of $X$ such that $\WF(u)$ is contained in the union of the co-normal bundles  $CN_{Y_i}^X$ of $Y_i\subset X$.}
\end{defn}

In \cite{AizDr}, the question is considered when the Fourier transform of a $\WF$-holonomic distribution is again a $\WF$-holonomic distribution, and for some distributions this property is shown. These distributions of \cite{AizDr} are all of $\cCexp$-class (up to working with charts to make the varieties affine), and it seems sensible to explore the mentioned question of \cite{AizDr} in the context of distributions of $\cCexp$-class, where more geometric tools are available than for abstract distributions, and which are stable under Fourier transforms by Theorem \ref{Fourier:p}.
The following generalization of Theorem A of \cite{AizDr} to general distributions of $\cCexp$-class may be a first step towards this question. 

\begin{thm}\label{thm:AizDrthmA}
Any $\cCexp$-class distribution is smooth on the complement of a proper Zariski closed subset. In more detail, let $Y$ and $W\subset Y\times \VF^m$ be definable sets for some $m\geq 0$ such that $W_{F,y}$ is a strict $C^1$ manifold of dimension $n$ for each $F$, each $y\in Y_F$ and some $n\leq m$. Let $E$ be in $\cCexp(W\times \ZZ)$. Then there are $G$ in $\cCexp(W)$ and a definable set $C\subset Y\times \VF^m$ such that for each $F$ in $\Loc'_{\gg 1}$ and for each $y\in Y_F$, the set $C_{F,y}$ is Zariski closed in $F^m$ of dimension at most $n-1$, the function $G_{F}(y,\cdot)$ is locally constant on $W_{F,y}\setminus C_{F,y}$, and, for $y\in {\rm Dis}(E,Y)_F$,
the distribution $u_{F,y}$ associated to $E_{F,y}$ is smooth on $W_{F,y}\setminus C_{F,y}$ and moreover represented by $G_{F}(y,\cdot)$ on $W_{F,y}\setminus C_{F,y}$. Moreover, in the case that $Y$ is a definable subset of $\VG^N\times \RF_{N}^N$ for some $N\geq 0$, there is a single Zariski closed subset $C$ of $\AA^{m}_\ZZ$ such that one can take for the $C_{F,y}$ the set $C(F)$ (independently from $y$).
\end{thm}
\begin{proof}
Let $G$ be in $\cCexp(W)$ such that $G_F(y,x)$ equals the limit for $r\to\infty$ of
$$
q_F^{nr}\cdot E_{F}(y,x,r)
$$
whenever this limit exists in $\CC$. Here, $F$ runs over $\Loc'_{\gg 1}$, $y$ over $Y_F$, and $x$ over $W_{F,y}$. Such $G$ exists by Theorem 3.1.3(2) of \cite{CGH5}.
 By Theorem 4.4.3 of \cite{CHallp}, the remark below its proof, by using charts (as at the end of the proof of Theorem \ref{Smo}), and by quantifier elimination result in the generalized Denef-Pas language from \cite{Rid}, there is a definable set $C_1 \subset Y \times \VF^m$ such that for every $F\in \Loc'_{\gg 1}$ and every $y\in Y_F$, the set $C_{1,F,y}$ is Zariski closed in $F^m$, has dimension at most $n-1$, and such that $G_{F}(y,\cdot)$ is locally constant on $W_{F,y}\setminus C_{1,F,y}$.
Similarly, there is a definable set $C_2 \subset Y \times \VF^m$ such that for every $F\in \Loc'_{\gg 1}$ and every $y\in Y_F$, the set $C_{2,F,y}$ is Zariski closed in $F^m$, has dimension at most $n-1$, and such that the function sending $r\gg 1$ to $E_{F}(y,x,r)$ is locally independent of $x$ in $W_{F,y}\setminus C_{2,F,y}$. By additivity of distributions (and since $E_{F,y}$ is a $B$-function), this implies, for $y\in {\rm Dis}(E,Y)_F$,  $x\in W_{F,y}\setminus C_{2,F,y}$ and for sufficiently large $r$, that
$$
E_{F}(y,x,r+1) = q_F^{n}\cdot E_{F}(y,x,r),
$$
and hence the above limit for $r\to\infty$  exists for all $x\in W_{F,y}\setminus C_{2,F,y}$. Now take $C_{F,y}$ to be the union of $C_{1,F,y}$ and $C_{2,F,y}$ to finish the proof, where the final simplified form for the $C_{f,y}$ follows again from the valued field quantifier elimination in the generalized Denef-Pas language from \cite{Rid}.
\end{proof}

\blue{
\subsection{Loci, their complements, and dimension}

Note that $\cCexp$-loci (and also complements of $\cCexp$-loci) are more general than definable sets, and that their study is more subtle than that of definable sets (the latter situation is well understood by the quantifier elimination result for definable sets in the generalized Denef-Pas language from \cite{Rid}, see also Theorem 5.1.2 of \cite{CHallp}, generalizing \cite{Pas} and \cite{Denef2}).
We continue this section with a study of dimensions of $\cCexp$-loci and their complements.

\begin{def-prop}[Dimensions of loci]\label{prop:localstructure}
Let $f$ be in $\cCexp(\VF^n)$. Fix $F$ in $\Loc'_{\gg 1}$. Write $A_1\subset F^n$ for the zero locus of $f_F$, and, $A_2$ for the complement of $A_1$ in $F^n$. Then there exist finitely many strict $C^1$-submanifolds $W_i$ of $F^n$ such that $A_1$ (resp.~$A_2$)  equals  $\bigcup_i W_i$. Moreover, for each $x\in W_i$ for any $i$, there exists an open neighborhood $U$ of $x$ in $F^n$ such that $W_i\cap U$ is $\gLPas(F)$-definable. Let $m$ be the maximum of the dimensions of the $W_i$, then we call $m$ the dimension of $A_1$ (resp.~of $A_2$). Finally, the Zariski closure of $\bigcup_i W_i$ in $F^n$ has dimension $m$ as well.
\end{def-prop}
\begin{proof}
Given $f$ and $F$, there are finitely many $\gLPas(F)$-definable (that is, definable with parameters from $F$) pieces $D_i\subset F^n$ whose union equals $F^n$, such that each $D_i$ is a strict $C^1$ manifold, and such that the restriction of $f_F$ to $D_i$ is locally constant. (This follows from cell decomposition and the definition of functions of $\cCexp$-class, and it requires $F$ to be fixed.)
Take $x\in A_1\cap D_i$ (resp.~in $A_2\cap D_i$). Let us put $W_i$ equal to $A_1\cap D_i$ (resp.~$A_2\cap D_i$). Since $f_F$ is locally constant on $D_i$, and since $D_i$ is a strict $C^1$ manifold, the $W_i$ are  clearly strict $C^1$ submanifolds. Similarly, for any $x\in W_i$, one clearly has that there exists an open neighborhood $U$ of $x$ in $F^n$ such that $W_i\cap U$ is definable. Again by the local constancy on $D_i$, either $W_i$ is empty, or, the dimensions of $D_i$ and $W_i$ coincide. Since $D_i$ is $\gLPas(F)$-definable, it is contained in a Zariski closed of the same dimension by the quantifier elimination result for definable sets in the generalized Denef-Pas language from \cite{Rid}.
\end{proof}
}

\begin{prop}\label{Zar}
Let $g$ be a function in $\cCexp(X)$ for some definable set $X\subset Y\times \VF^\ell$, where $Y$ is a definable set, and let \blue{$k \le \ell$} be given. For each $F$ in $\Loc'_{\gg 1}$, write $A_{F}$ for the complement of the zero locus of $g_F$ in $X_F$. For each $y\in Y_F$, write $Z_{F,y}$ for the Zariski closure of $A_{F,y}$ in $F^\ell$.
Then there exists a definable set $C\subset Y\times \VF^\ell$ such that for each $F$ in $\Loc'_{\gg 1}$ and for each $y\in Y_F$, the set $C_{F,y}$ is Zariski closed in $F^\ell$ of dimension at most $k$, and if $\dim Z_{F,y} \le k$, then $Z_{F,y} \subset C_{F,y}$.
In particular, the set of $y$ for which we have $\dim Z_{F,y} \le k$ is a $\cCexp$-locus.
\end{prop}



\begin{proof}
We may assume that $X = Y \times \VF^\ell$, by extending $g$ by $0$ on $(Y \times \VF^\ell) \setminus X$.

First note that the ``in particular'' part indeed follows from the remainder of the proposition:
From $Z_{F,y} \subset C_{F,y} \iff \dim Z_{F,y} \le k$,
we obtain that $\dim Z_{F,y} \le k$ iff $g$ vanishes on the definable set $F^\ell \setminus C_{F,y}$. The set of $y$ for which this holds
is a $\cCexp$-locus by \cite[Proposition 1.3.1]{CGH5}.

Now consider the main part of the proposition.
For $k = m$, there is nothing to prove. (Set $C := Y \times \VF^\ell$.)

In the case $k = \ell - 1$, we use
Theorem 4.4.3 of \cite{CHallp} and the remark below its proof to find a definable set $C \subset Y \times \VF^\ell$ such that
for every $F\in \Loc'_{\gg 1}$ and every $y\in Y_F$, the set $C_{F,y}$ is Zariski closed in $F^\ell$, has dimension at most $\ell-1$, and
contains the set of non-local constancy of $g_{F,y}$.
Then clearly, if $Z_{F,y}$ has dimension at most $k$, it is contained in $C_{F,y}$.

For $k < \ell - 1$, we use the case $k=\ell-1$ to find a definable set $C_{\ell-1}\subset Y \times \VF^\ell$ and one finishes the proof by induction on $\ell$, by using the restriction of $g$ to $C_{\ell-1}$ and charts on $C_{\ell-1}$ which are coordinate projections (as at the end of the proof of Theorem \ref{Smo}), and using the quantifier elimination result in the generalized Denef-Pas language from \cite{Rid} to get to definable sets which are Zariski closed as desired.
\end{proof}

\subsection{Holonomicity}

The following result relates to Remark 3.2.2 of \cite{AizDr}, \blue{and is about witnessing algebraic holonomicity. Its first part states} that given a family of distributions $u_{F,y}$,
the set of $F$ and $y$ such that $u_{F,y}$ is \blue{strict $C^1$} $\WF$-holonomic is a locus sets. Further, the $u_{F,y}$ that are $\WF$-holonomic are even ``uniformly algebraically $\WF$-holonomic'' in the sense that their wave front sets are contained in \blue{a finite union of the co-normal bundles of uniformly definable algebraic submanifolds.
 Finally, it says that, for $\cCexp$-class distributions, to be algebraically $\WF$-holonomic is essentially the same as being strict $C^1$ $\WF$-holonomic.}





%
%

\begin{prop}\label{thm:WFh}
Let $Y$ and $W\subset Y\times \VF^n$ be definable sets for some $n\geq 0$. Let $E$ be in $\cCexp(W\times \ZZ)$. For each $F$ in $\Loc'_{\gg 1}$ and each $y\in {\rm Dis}(E,Y)_F$, denote the distribution on $W_{F,y}$ associated to $E_{F,y}$ by $u_{F,y}$.
Define the sets, for $F$ in $\Loc'_{\gg 1}$,
$$
{\rm{Hol}}(E,Y)_F := \{y\in Y_F \mid y\in {\rm Dis}(E,Y)_F \mbox{  and }
\mbox{$u_{F,y}$ is \blue{strict $C^1$} $\WF$-holonomic} \}.
$$
Then ${\rm{Hol}}(E,Y)$ is a $\cCexp$-locus. Moreover, \blue{there exist finitely many definable subsets $X_i\subset W$ such that for each $F$ in $\Loc'_{\gg 1}$ and each $y\in {\rm Dis}(E,Y)_F$, each $X_{i,F,y}$ is the set of $F$-rational points on a smooth locally closed subvariety 
of $\AA_F^{n}$, $X_{0,F,y}$ has the same dimension as $W_{F,y}$, $X_{0,F,y}$ contains both $W_{F,y}$ and the $X_{i,F,y}$, and, $u_{F,y}$ is strict $C^1$ $\WF$-holonomic if and only if the $F^\times$-wave front set of $u_{F,y}$ is contained in
$$
\bigcup_i CN_{X_{i,F,y}}^{X_{0,F,y}}.
$$
Hence, for any $F$ in $\Loc'_{\gg 1}$ and any $y\in {\rm Dis}(E,Y)_F$,
the distribution $u_{F,y}$ on $W_{F,y}$ is strict $C^1$ $\WF$-holonomic if and only if it can be extended to an algebraically $\WF$-holonomic distribution on $\cX(F)\supset W_{F,y}$ for some smooth algebraic subvariety $\cX$ of $\AA^n_F$ of the same dimension as $W_{F,y}$.
}
\end{prop}
\begin{proof}
\blue{
By Theorem \ref{Smo}, $\Smo(\Lambda,E,Y)$ is a $\cCexp$-locus, say, given as the locus of $g\in \cCexp (W\times  \VF^n)$.  We may suppose that the dimension of $W_{F,y}$ equals $k$ for each $F$ and each $y\in  {\rm Dis}(E,Y)_F$.
Let $C$ be as in Proposition \ref{Zar} for $g$ and $k$, and, let $C_0$ be as in Proposition \ref{Zar} for the characteristic function of $W$ and $k$.
By Proposition \ref{prop:localstructure}, Chevalley's theorem and valued field quantifier elimination in the generalized Denef-Pas language, there are finitely many definable sets $X_{i}$ as postulated in the proposition. Namely, each $X_{i,F,y}$ equals the set of $F$-rational points on a locally closed smooth submanifold of $\AA_F^{n}$ of dimension at most $k$, and, for each $F$ and each $y\in  {\rm Dis}(E,Y)_F$, the set $X_{0,F,y}$ contains both $W_{F,y}$ and the $X_{i,F,y}$, and, $u_{F,y}$ is strict $C^1$ $\WF$-holonomic if and only if $\WF(u_{F,y})$ is contained in
$$
\bigcup_i CN_{X_{i,F,y}}^{X_{0,F,y}}.
$$
By the formalism of \cite[Proposition 1.3.1]{CGH5}, this condition corresponds to a $\cCexp$-locus which finishes the proof.
}
\end{proof}

We now get a transfer principle for $\WF$-holonomicity.

\begin{cor}[Transfer Principle for $\WF$-holonomicity]\label{transfer}
Let data and notation be as in Proposition~\ref{thm:WFh}. Then there is $M$ such that for all $F_1,F_2\in \Loc_M$ with isomorphic residue fields, one has that
$$
Y_{F_1} = {\rm{Hol}}(E,Y)_{F_1,\psi} \mbox{ for all $\psi$ in $\cD_{F_1}$}
$$
if and only if
$$
Y_{F_2} = {\rm{Hol}}(E,Y)_{F_2,\psi} \mbox{ for all $\psi$ in $\cD_{F_2}$}
$$
\end{cor}
\begin{proof}
By Proposition~\ref{thm:WFh}, to be $\WF$-holonomic is a $\cCexp$-locus condition. By Proposition 9.2.1 of \cite{CLexp}, one has transfer for any $\cCexp$-locus condition.
\end{proof}

\begin{remark}
\blue{All results, statements and definitions of Sections \ref{sec:distCexpclass} and \ref{sec:WF} (modulo adaptations in the parts concerning algebraic varieties) hold when one consequently replaces $\gLPas$ by an enrichment obtained by adding some analytic structure as in \cite{CLip}, \cite{CLips} to $\gLPas$. Similarly, one can put additional structure on the residue rings $\Res_n$, and, one can add constants for a ring of integers $\cO$ of a number field in the sort $\VF$ and work uniformly in all finite field extensions of completions of the fraction field of $\cO$.  Indeed, this corresponds to Remark A.3 of \cite{CGH5} and Section 4.7 of \cite{CHallp}.}  
\end{remark}

\section{Discriminants and Schwartz-Bruhat functions}\label{sec:DiscrSB}

\subsection{}
Let $X$ be a smooth projective variety defined over $\QQ$.
By abuse of notation we consider $X$ as a definable set by taking for $X_F$ the set of $F$-rational points $X(F)$ on $X$ for all $F$ in $\Loc_{\gg 1}$.
More precisely there exists a finite family of definable sets $X_i$, $1 \leq i \leq n$, such
that for any  $F$ in $\Loc_{\gg 1}$, $X (F)= \bigcup_i X_{i, F}$ and each $X_{i, F}$ is open in $X (F)$.
Consider a definable function $d$ with, for each $F\in \Loc_{\gg 1}$,
$$
d_F :  X(F) \times X (F) \to \ZZ \cup \{+\infty\}
$$
satisfying that
$d_F (x, x') = +\infty$ if and only if $x = x'$,  $d_F (x, x') =d_F (x', x)$
and
$d_F (x,z) \geq \min (d_F(x,y),d_F(y,z))$.
For any $x_0 \in X(F)$ and any integer $m$, we denote  by  $B_{F,x_{0}, m}$ the subset of $X(F)$ consisting of points $x$ such that $d_F(x,x_0)\geq m$, and we call such a set a ball of valuative radius $m$ in $X (F)$.

We call $d$ a definable (valuative) metric on $X$ if, for all $F\in \Loc_{\gg 1}$, all the balls $B_{F,x_{0}, m}$ are open
and if they generate the valuation topology on $X (F)$. Assume now that $d$ is a definable metric on $X$.

\subsection{}
Let $D$ be a divisor in $X$.
We denote by $T(D (F), \varepsilon)$ the tube of valuative distance $\varepsilon$ around $D (F)$, namely,
$$
\{x\in X(F)\mid d_F (x, D (F))>\varepsilon \},
$$
where
$$
d_F(x, D (F)) = \sup_{y \in D (F)} d_F(x,y).
$$
Let us write $X (D (F), \varepsilon)$ for the set $X (F)\setminus T (D (F), \varepsilon)$.

For any open subset  $W$ of $X(F)$ and
any integer $m$, let $\cS_m (W)$ denote
the set of $\mathbb C$-valued functions on $W$ that are constant on $B\cap W$ for all
balls $B$ of valuative radius $m$ in $X(F)$. 



Fix now an affine function  $\lambda :   \NN \times \NN \to \NN$
with positive coefficients.

We denote by $\cS^{m}_\lambda (X (F),  D(F))$ the set of functions on $X (F) \setminus D (F)$
whose restriction to  $X (D (F), \varepsilon)$ belongs to
$\cS_{\lambda (\varepsilon, m)} (X (D (F), \varepsilon))$, for any $\varepsilon$.
Furthermore, we denote by $\cS^{m}_\lambda (X,  D)$ the set of functions $\varphi$ in $\cCexp(X\setminus D)$ such that $\varphi_F$ lies in $\cS^{m}_\lambda (X (F),  D(F))$ for each $F\in \Loc'_{\gg 1}$.

\subsection{}Let $f : X \to Y$ be a morphism of smooth projective varieties defined over $\QQ$. Fix a definable valuative metric $d'$ on $Y$.
We denote  by $\Delta_f$ the discriminant of
$f : X \to Y$, that is, the smallest closed subset $Z$ of $Y$ such that the restriction of the morphism $f$  to
$X \setminus f^{-1} (Z) \to Y \setminus Z$ is smooth.
For  $F\in \Loc_{\gg 1}$
we write $f_F:X(F)\to Y(F)$ for the corresponding map between  sets  of $F$-rational points.

Let $D$ be a divisor in
$X$ and let $\omega$ be  an algebraic volume form on $X_0 = X \setminus D$.
Let $D'$ be a divisor in $Y$ containing $f (D)$ and
 the discriminant
$\Delta_f$. Let $\omega'$ be  an algebraic volume form on $Y_0 = Y \setminus D'$.

For $\varphi$ a bounded function on $X_0(F)$, one denotes the function
$$
y \longmapsto \int_{x\in f_F^{-1} (y) \cap X_0 (F)} \varphi(x)\Bigl \vert \frac{\omega }{f^* \omega'} \Bigr\vert
$$
on $Y_0 (F)$ by
$f_{F,!} (\varphi)$.

It is easy to prove (using partitions of unity, the definable, strict $C^1$ Sard Lemma and Fubini) that if $\varphi$ is locally constant on $X_0 (F)$ then
$f_{F,!} (\varphi)$ is  locally constant on $Y_0 (F)$, see e.g.~Prop.~3.3.1 of \cite{AizAvni}.

By  Theorem 4.1.1 of \cite{CHallp},
 there exists a map
$f_! : \cCexp(X\setminus D) \to \cCexp(Y\setminus D')$
such that
$f_{F,!} (\varphi_F) =  (f_{!} (\varphi))_F$
for each $F\in \Loc'_{\gg 1}$ and any $\varphi$ in $\cCexp(X\setminus D)$.


%

\begin{thm}Let $f : X \to Y$ be a morphism of smooth projective varieties over $\QQ$.
Suppose that both, $X$ and $Y$, are equipped with a definable metric.
Let $D$ be a divisor in
$X$ and let $\omega$ be an algebraic volume form on $X_0 = X \setminus D$.
Let $D'$ be a divisor in $Y$ containing $f (D)$ and
 the discriminant
$\Delta_f$. Let $\omega'$ be  an algebraic volume form on $Y_0 = Y \setminus D'$.
For any  affine function  $\lambda :   \NN \times \NN \to \NN$
with positive coefficients, there exists an
affine function  $\mu :   \NN \times \NN \to \NN$
with positive coefficients, such that
for any integer $m$
and for
any $\varphi \in \cS^m_\lambda (X,  D)$,
$f_! (\varphi)$
belongs to  $\cS^m_\mu (Y,  D')$.
\end{thm}

\begin{proof}
For $x\in X$ and $r\in \ZZ$ write $\11_{x,r}$ for the characteristic function of the ball $B_{x,r}$ in $X$ containing $x$ and of valuative radius $r$.
Write $Y_\varepsilon$ for the definable set $Y \setminus T (D', \varepsilon)$ and write $g(x,\cdot,r,\varepsilon)$ for the restriction of $f_! (\11_{x,r}) := ( f_{!,F} (\11_{F,x,r}) ) _{F}$ to $Y_\varepsilon$. Note that by the stability under integration given by Theorem 4.1.1 of \cite{CHallp}, $f_! (\11_{x,r})$ is of  $\cCexp$-class.
By \cite[Corollary 1.4.3]{CGH5}, there is a definable function
$$
\alpha:X \times Y_{\varepsilon} \times \ZZ^2 \to\ZZ
$$
such that the restriction to $B_{y_0,\alpha(x,y_0,r,\varepsilon)}$ of
$y\mapsto g(x, y  ,r,\varepsilon)$
is constant for any $(x,y_0,r,\varepsilon)$ in $X\times Y_{\varepsilon}$. Here, $B_{y_0,\alpha(x,y_0,r,\varepsilon)}$ is the ball in $Y$ of radius $\alpha(x,y_0,r,\varepsilon)$ around $y_0$.
Clearly, we may suppose that $\alpha$ is continuous as a function in $y$. But then, by local compactness, it attains a finite maximum
$$
\beta(x,r,\varepsilon) := \max_{y\in Y_\varepsilon} \alpha(x,y,r,\varepsilon).
$$
Clearly, $\beta$ is definable and we may suppose that it is continuous as a function in $x$. Again by compactness, it attains a finite maximum
$$
\gamma(r,\varepsilon) := \max_{x\in X} \beta(x,r,\varepsilon).
$$
Clearly
$\gamma$ is definable, and hence, bounded by a linear function $\mu$ in $(\varepsilon,r)$, which is as desired.
\end{proof}

\bibliographystyle{amsplain}
\bibliography{anbib}

\end{document}